\documentclass{amsart}

\usepackage[english]{babel}
\usepackage[utf8]{inputenc}
\usepackage[T1]{fontenc}

\usepackage{amsmath,amsfonts,amssymb,amscd}
\usepackage{mathrsfs}
\usepackage{float,enumitem}
\usepackage{hyperref}
\usepackage{ifthen}
\usepackage{tikz}
\usepackage{tikz-cd}
\usepackage{graphicx}

\title[Convergence of locally symmetric spaces]{Benjamini--Schramm convergence of arithmetic locally symmetric spaces}

\author{Miko\l{}aj Fr{a}czyk}
\address{Jagiellonian University, Faculty of Mathematics and Computer Science, ul. {\L}ojasiewicza 6, 30-348 Krak{\'o}w, Poland }
\email{mikolaj.fraczyk@uj.edu.pl}
\author{Sebastian Hurtado}
\address{Yale University, Department of Mathematics, 10 Hillhouse Ave, New Haven, CT 06511}
\email{sebastian.hurtado-salazar@yale.edu}
\author{Jean Raimbault}
\address{Institut de Mathématiques de Marseille, UMR 7373, CNRS, Aix-Marseille Université}
\email{jean.RAIMBAULT@univ-amu.fr}

\DeclareFontFamily{U}{wncy}{}
\DeclareFontShape{U}{wncy}{m}{n}{<->wncyr10}{}
\DeclareSymbolFont{mcy}{U}{wncy}{m}{n}
\DeclareMathSymbol{\Sha}{\mathord}{mcy}{"58}

\newcommand{\liea}{\mathfrak{a}}

\newcommand{\sub}{\mathrm{Sub}}
\newcommand{\eps}{\varepsilon}

\newcommand{\wdt}[1]{\widetilde #1}

\newcommand{\ovl}[1]{\overline #1}

\newcommand{\bs}{\backslash}

\newcommand{\vol}{\operatorname{vol}}
\newcommand{\Vol}{\operatorname{vol}}

\newcommand{\tr}{\operatorname{tr}}

\newcommand{\rk}{\operatorname{rk}}

\newcommand{\Aut}{\operatorname{Aut}}
\renewcommand{\hom}{\operatorname{Hom}}
\newcommand{\Hom}{\operatorname{Hom}}
\newcommand{\inj}{\operatorname{inj}}

\newcommand{\id}{\operatorname{Id}}
\newcommand{\ad}{\operatorname{Ad}}
\newcommand{\Gal}{\operatorname{Gal}}

\newcommand{\Fix}{\operatorname{Fix}}

\newcommand{\interval}[4]{
  \ifthenelse{ \equal{#1}{o} } {\mathopen{]}} {\mathopen{[}}
  #2, #3
  \ifthenelse{ \equal{#4}{o} } {\mathclose{[}} {\mathclose{]}}
}

\newcommand{\T}{\mathbf{T}}
\renewcommand{\H}{\mathbf{H}}

\newcommand{\V}{\mathbf{V}}
\newcommand{\Z}{\mathbf{Z}}
\newcommand{\N}{\mathbf{N}}
\newcommand{\G}{\mathbf{G}}

\newcommand{\SL}{\mathrm{SL}}

\newcommand{\SO}{\mathrm{SO}}

\newcommand{\GL}{\mathrm{GL}}
\newcommand{\PGL}{\mathrm{PGL}}

\newcommand{\val}{\mathrm{V}}

\newcommand{\bG}{\mathbf{G}}

\newcommand{\mo}{\mathfrak{o}}

\newcommand{\res}{\mathrm{res}}

\newcommand{\Stab}{\mathrm{Stab}}

\newcommand{\CC}{\mathbb C}
\newcommand{\RR}{\mathbb R}
\newcommand{\ZZ}{\mathbb Z}

\newcommand{\QQ}{\mathbb Q}

\newcommand{\NN}{\mathbb N}

\newcommand{\Ade}{\mathbb A}

\newcommand{\InjRad}{\text{InjRad}}

\newcommand{\B}{\mathrm{B}}

\numberwithin{equation}{section}

\begin{document}

\newtheorem{theorem}{Theorem}[section]
\newtheorem{lemma}[theorem]{Lemma}
\newtheorem{corollary}[theorem]{Corollary}
\newtheorem{proposition}[theorem]{Proposition}
\newtheorem{conjecture}[theorem]{Conjecture}
\newtheorem{question}[theorem]{Question}
\newtheorem{remark}[theorem]{Remark}
\newtheorem{definition}[theorem]{Definition}

\newtheorem{theostar}{Theorem}
\renewcommand*{\thetheostar}{\Alph{theostar}}

\newtheorem{innertheoremrep}{Theorem}
\newenvironment{theoremrep}[1]
  {\renewcommand\theinnertheoremrep{#1}\innertheoremrep}
  {\endinnertheoremrep}

\maketitle

\begin{abstract}
  We prove that the thin parts of arithmetically defined locally symmetric space take up a negligible part of their volume and deduce asymptotic results on their Betti numbers. 
\end{abstract}

\setcounter{tocdepth}{1}
\tableofcontents


\section{Introduction}

In \cite{7Sam} the notion of Benjamini--Schramm convergence for sequences of locally symmetric spaces of finite volume was introduced, and in \cite{Abert_Biringer} it was generalised to Riemannian manifolds. This gives a compactification for the space of finite-volume quotients of a symmetric space $X$, at the cost of introducing limit points which are in general not manifolds themselves but instead {\em unimodular random manifolds}: random pointed manifolds satisfying an additional homogeneity condition. The simplest examples of such are finite-volume manifolds (the root can be chosen randomly according to the normalised volume measure) and symmetric spaces (there is a unique root). In the setting of locally symmetric spaces an almost equivalent notion (via the correspondence between discrete subgroups of Lie groups and locally symmetric spaces) is that of an {\em invariant random subgroup} of the isometry group, which is simply an invariant probability distribution on the Chabauty space of closed subgroups. Convergence in this space is usually called {\em Benjamini--Schramm convergence}. 

In the context of graphs Benjamini--Schramm convergence was introduced in \cite{BeSc01}; the relation with invariant random subgroups (in discrete groups) was introduced in \cite{AGV14}. A survey of the topic with emphasis on the locally symmetric setting is given in \cite{Gelander_ICM}. Concretely, for a sequence of manifolds $M_n$ of finite volume, locally isometric to a symmetric space $X$, Benjamini--Schramm convergence to $X$ (henceforth abbreviated as "BS-convergence") is equivalent to the following condition: 
\begin{equation} \label{BSconv_univcover}
  \forall R > 0 \: \lim_{n \to +\infty} \frac{\vol((M_n)_{\le R})}{\vol M_n} = 0 
\end{equation}
where for a Riemannian manifold $M$ we denote by $M_{\le R}$ its $R$-thin part, that is the subset of points where the injectivity radius is at most $R$. (For an orbifold we can also define the $R$-thin part, which will include the singular locus.)

It follows from the Stuck--Zimmer theorem \cite{Stuck_Zimmer} that, when $X$ is of higher rank without real- or complex-hyperbolic factors, BS-convergence to $X$ always holds for proper sequences of irreducible manifolds (see \cite[Theorem 1.5]{7Sam}). Besides the higher-rank case, a far-reaching example of such convergence taking place is that of congruence covers of a fixed manifold \cite[Theorem 1.12]{7Sam}. In \cite{Fraczyk} the first author showed that for real-hyperbolic 2- and 3-manifolds this is also true for sequences of non-commensurable arithmetic manifolds. The main result in the present paper is the following generalisation of this result to all symmetric spaces $X$ (see \ref{prelim_lattices} for the general definition of congruence lattices). 

\begin{theostar} \label{Main_BSconv}
  Let $G$ be a noncompact semisimple Lie group, let $X$ be its symmetric space and let $\Gamma_n$ be a sequence of pairwise non-conjugate congruence arithmetic lattices. Then, \eqref{BSconv_univcover} holds for the sequence of locally symmetric spaces $\Gamma_n \bs X$. 
\end{theostar}

We note that this theorem remains valid when taking a sequence of arithmetic lattices which are not necessarily congruence but which contains no infinite subsequence of pairwise commensurable lattices (or even no infinite sequence of lattices contained in the same maximal arithmetic lattice), and that this apparently more general statement follows formally from the theorem stated above since it is well-known that maximal arithmetic lattices are congruence (see Proposition \ref{defn_congruence} below). In fact we work directly with maximal lattices in our proof of Theorem \ref{Main_BSconv}. 

One of the main motivations for introducing Benjamini--Schramm convergence is that when a sequence of finite-volume $X$-manifolds converges to $X$, one can deduce asymptotic bounds for its Betti numbers (and more generally for multiplicities of automorphic representations). A very general result for negatively curved manifolds is given by Abért--Bergeron--Biringer--Gelander in \cite[Corollary 1.4]{ABBG_bettinumbers}. As an immediate corollary of their result and Theorem \ref{Main_BSconv} we get the following. 

\begin{theostar} \label{betti_numbers}
  Let $\Gamma_n$ be a sequence of pairwise distinct, torsion-free irreducible congruence arithmetic lattices or pairwise non-commensurable irreducible arithmetic lattices in a semisimple Lie group $G$. Let $\beta_i^{(2)}(X)$ be the $L^2$-Betti numbers of the symmetric space $X$ (see \ref{L2Betti} below for a quick description). Then we have: 
  \[
  \lim_{n\to\infty} \frac{ b_i(\Gamma_n\bs X)}{\vol(\Gamma_n\bs X)} = \begin{cases} \beta_i^{(2)}(X) \text{ if } i = \dim(X)/2;  \\ 0 \text{ otherwise. }\end{cases}
  \]
\end{theostar}

When all $\Gamma_n$ are uniform we do not need the general result of \cite{ABBG_bettinumbers} to prove this, as we can deduce it from more general results on limit multiplicities for automorphic representations. This is explained in detail in \ref{lim_mult_intro} just below. In the case where the degree of trace fields goes to infinity, we have more precise estimates on the geometric convergence from our previous work \cite[Theorem D]{FHR_complexity} (see Theorem \ref{unbounded_degree} below) and we can deduce more precise bounds on Betti numbers away from the middle degree. 

\begin{theostar} \label{quant_betti_numbers}
  Let $i\in 0,1,\ldots, \dim X$. Then 
  \[
  b_i(\Gamma\bs X)= \vol(\Gamma\bs X) \cdot \left(\beta_i^{(2)}(X) + O([k:\QQ]^{-1}) \right)
  \]
  for all uniform arithmetic lattices $\Gamma \subset G$. 
\end{theostar}

It is worth pointing out that this result is interesting only when $[k:\QQ]\to\infty$, which excludes non-uniform lattices.  


\subsubsection{The proof of Theorem \ref{Main_BSconv}}

We prove Theorem \ref{Main_BSconv} by first distinguishing between the cases when the arithmetic lattices are defined using number fields of bounded or unbounded degree over $\QQ$ (see \ref{prelim_lattices} below for the definition of the trace field). We dealt with the latter case in a previous paper where we established the following quantitative result. 

\begin{theostar}{\cite[Theorem D]{FHR_complexity}} \label{unbounded_degree}
  Let $G$ be a semisimple Lie group with associated symmetric space $X$. There are positive constants $c=c_X,\eta=\eta_X$ such that any arithmetic lattice $\Gamma\subset G$ with trace field $k$ satisfies 
  \[
  \vol((\Gamma\bs X)_{\le \eta[k:\QQ]}) \le \vol(\Gamma\bs X)e^{-c[k:\QQ]}.
  \]
\end{theostar}

In this paper we deal with (most of) the remaining cases in the proof of Theorem \ref{Main_BSconv} by establishing the following result. 

\begin{theostar} \label{Main_bounded_degree}
  Let $G$ be a noncompact semisimple Lie group and let $X$ be its symmetric space. Let $\Gamma_n$ in $G$ be a sequence of pairwise distinct maximal arithmetic lattices in $G$ and assume that there exists $d \in \NN$ such that the trace field of each $\Gamma_n$ is an extension of degree $d$ of $\mathbb Q$. Then \eqref{BSconv_univcover} holds for the sequence of locally symmetric spaces $\Gamma_n \bs X$. 
\end{theostar}

From these two results and another one from \cite{7Sam} we can immediately deduce the general statement in Theorem \ref{Main_BSconv}. For convenience we detail the argument here: let $\Gamma_n$ be a sequence of congruence arithmetic lattices in $G$. Passing to a sub-sequence we can assume that 
\begin{equation} \label{cond_largethin}
  \lim_{n\to\infty}\frac{\vol((\Gamma_n\bs X)_{\le R})}{\vol (\Gamma_n\bs X)} = \ell, 
\end{equation}
for some $\ell\geq 0$. Theorem \ref{Main_BSconv} will follow once we show that any such limit $\ell$ is zero. Let $k_n$ be the trace field of $\Gamma_n$ and let $d_n=[k_n:\QQ].$ If $d_n$ are unbounded, then $\ell=0$ by Theorem \ref{unbounded_degree}. If $d_n$ are bounded, choose a sequence of maximal arithmetic lattices $\Delta_n$, such that $\Gamma_n\subset \Delta_{n}$ for every $n$. If the set of $\Delta_n$ is finite then $\ell=0$ by  \cite[Theorem 1.12]{7Sam}. Finally, if the sequence $\Delta_n$ is infinite, we can apply Theorem \ref{Main_bounded_degree} to show 
\[
\ell=\lim_{n\to\infty}\frac{\vol((\Gamma_n\bs X)_{\le R})}{\vol (\Gamma_n\bs X)}\leq \lim_{n\to\infty}\frac{\vol((\Delta_n\bs X)_{\le R})}{\vol (\Delta_n\bs X)}=0
\]
which concludes the proof of Theorem \ref{Main_BSconv}. 

\medskip

Most of the paper is dedicated to establish Theorem \ref{Main_bounded_degree}. Since \cite[Theorem D]{FHR_complexity} is much more precise than what we need for \ref{Main_BSconv}, we will also give a shorter proof of a non-quantitative statement for lattices with unbounded trace field degree (see \ref{s:cheat_mode} below). 


\subsubsection{Outline of the proof of Theorem \ref{Main_bounded_degree}}

The proof of  Theorem \ref{Main_bounded_degree} in this paper is completely different from that given in \cite{FHR_complexity} for Theorem \ref{unbounded_degree}. It follows the proof for groups of type $A_1$ given \cite{Fraczyk}, estimating the geometric side of Selberg's trace formula, though we were not able to get quantitative results for all terms in the higher-rank case. To conclude we have to resort to a non-effective argument similar to that in \cite{FraRaim}.

As in \cite{Fraczyk}, we express the volume of the thin part as the trace of a certain convolution operator. Recall that $X=G/K$, where $K$ is a maximal compact subgroup of $G$. There is a unique bi-$K$-invariant semi-metric on $G$ such the image of an $R$-ball in $G$ is precisely an $R$-ball in $X$. Let $1_R\colon G\to \RR$ be the characteristic function of the ball of radius $R$ in $G$. A point $\Gamma g K\in \Gamma\bs X$ is in the $R$-thin part if and only if $\gamma gK$ is in the $R$-ball around $gK$, for some non-central $\gamma\in \Gamma$. Therefore $\Gamma g K\in \Gamma\bs X\in (\Gamma\bs X)_{\leq R}$ if and only if $\sum_{\gamma \in \Gamma\setminus Z(\Gamma)} 1_R(g^{-1}\gamma g)>0$. Benjamini-Schramm convergence will follow once we show that the integral
\[
\frac 1 {\vol(\Gamma_n \bs G)} \int_{G/ \Gamma_n}  \sum_{\gamma \in \Gamma_n\setminus Z(\Gamma)} 1_R(g^{-1}\gamma g) dg
\]
converges to zero as $n \to \infty$. Using the habitual manipulations for the geometric side of Selberg's trace formula for compact quotient (see \cite[I.1]{Clay03}) this integral is equal to 
\[
\frac 1 {\vol(\Gamma_n \bs G)} \sum_{[\gamma]_{\Gamma_n} \subset W} \vol((\Gamma_n)_\gamma\bs G_\gamma)\mathcal O(\gamma, 1_R),
\]
where $[\gamma]_{\Gamma_n}$ is the conjugacy class of $\gamma$ in $\Gamma_n$, $O(\gamma,1_R)$ is the orbital integral
\[
\mathcal O(\gamma, f) = \int_{G_\gamma \bs G} 1_R(x\gamma x^{-1}) dx
\]
and $W$ is the set of non-central elements of $\Gamma_n$. 
 
The conjugacy classes of elements in $\Gamma_n$ are difficult to parametrise, so instead we use the adelic trace formula as in \cite{Fraczyk}(see \ref{adelic_reformulation}). The adelic trace formula replaces the sum over the conjugacy classes in $\Gamma_n$ by a sum over the rational conjugacy classes in certain semi-simple algebraic group defined over the trace field of $\Gamma_n$. The rational conjugacy classes are much easier to parametrise\footnote{For example, classifying conjugacy classes in $\SL_n(\ZZ)$ is more difficult than classifying conjugacy classes in $\SL _n(\QQ)$} and the adelic orbital integrals are in a sense better behaved. Giving bounds for this summation of integrals in the adelic setting is the main contribution from this part of our work and the technical parts take up various sections. The main ingredients are known estimates for the volumes of adelic quotients of tori (mostly from \cite{Ullmo_Yafaev}), strong bounds for local and global orbital integrals (given in Section \ref{s:orbint}) and estimates on the number of rational conjugacy classes of elements of small Weil height (given in Section \ref{s:small}). 
These bounds are quite delicate and difficult to obtain for conjugacy classes with large centralizers. Using a reduction argument adapted from \cite{FraRaim} (see Theorem \ref{general_criterion} and Lemma \ref{suff_dense} for precise statements) we can restrict $W$ to be the set of highly regular conjugacy classes and still get the conclusion about the Benjamini-Schramm convergence. This part of the argument uses the Zariski density of invariant random subgroups proved in \cite{7Sam}.

Because of the reduction step, the argument does not give an effective bound on the volume of the $R$-thin part of the $\Gamma_n\bs X$, since we completely sidestep estimating the orbital integrals and covolumes of centralisers for irregular semisimple elements. 


\subsection{$L^2$-Betti numbers} \label{L2Betti}

The general theory of $L^2$-Betti numbers is described in the treatise \cite{Lueck_book}. They are usually defined for compact Riemannian manifolds or finite CW-complexes. In the case of locally symmetric spaces there is a proportionality principle which states that for a symmetric space $X$ and a finite-volume quotient $M$ of $X$ the $L^2$-Betti number $b_i^{(2)}(M)$ (computed for a triangulation of the Borel--Serre compactification) are given by $\vol(M) \beta_i^{(2)}(X)$ for a constant $\beta_i^{(2)}(X)$. 

The $L^2$-Betti numbers $\beta_i^{(2)}(X)$ can be nonzero only when $i=\dim(X)/2$, and moreover $\beta_{\dim(X)/2}^{(2)}(X)$ can be computed explicitly: it is equal to $\chi(X^d)/\vol(X^d)$ where $X^d$ is the compact dual of $G$ endowed with the Haar measure compatible with that of $G$. A list of all simple Lie groups $G$ (up to isogeny) of non-compact type for which $\beta_{\dim(X)/2}^{(2)}(G) \not= 0$ is as follows:
\begin{itemize}
\item all unitary groups $\mathrm{SU}(n, m)$;

\item all orthogonal groups $\mathrm{SO}(n, m)$ with $nm$ even and the groups $\mathrm{SO}^*(2n)$;

\item all symplectic groups $\mathrm{Sp}(n)$ and all groups $\mathrm{Sp}(p, q)$ (isometries of quaternionic hermitian forms);

\item some exceptional groups (at least one in each absolute type). 
\end{itemize}
This is well--known, and can be recovered as follows: by \cite[Theorem 1.1]{olbrich}, a group is on the list if and only if it has a compact maximal torus. In terms of the classification by Vogan diagrams \cite[VI.8]{Knapp_book_lie} this means that the action of the Cartan involution on the underlying Dynkin diagram is trivial. The list of such diagrams is given in this book, in Figure 6.1, p.~414 and Figure 6.2, p.~416 for classical and exceptional types respectively.


\subsection{Limit multiplicities} \label{lim_mult_intro}

Using Matsushima's formula we can deduce both theorems above from the results on limit multiplicities. To explain this in detail we need to introduce some notations. If $\pi$ is a unitary representation of $G$, acting on a Hilbert space $\mathcal H_\pi$, and $\Gamma$ a lattice $G$, then $m(\pi, \Gamma)$ is the {\em multiplicity} of $\pi$ in $L^2(\Gamma \bs G)$, that is 
\[
m(\pi,\Gamma)=\dim_{\CC} \Hom(\pi, L^2(\Gamma\bs G)).
\]
If $\Gamma$ is uniform then the multiplicity $m(\pi,\Gamma)$ is finite, and nonzero for only countably many $\pi$. Moreover $L^2(\Gamma \bs G) \cong \overline\bigoplus_\pi \mathcal H_\pi^{m(\pi, \Gamma)}$ with the sum running over all isomorphism classes of unitary representations of $G$.

The problem of {\em limit multiplicities} in its most basic form asks whether for a sequence of (pairwise non-conjugated) lattices $\Gamma_n$in $G$ we have
\begin{equation} \label{limmult}
  \frac{m(\pi, \Gamma_n)}{\vol(\Gamma_n \bs G)} \xrightarrow{n \to +\infty} d(\pi)
\end{equation}
for all $\pi$, where $d(\pi)$ is the ``formal degree'' of $\pi$, which is nonzero unless $\pi$ is a discrete series \cite{Knapp2}. 
On the other hand {\em Matsushima's formula} \cite{Matsushima} states the for every $i$ there is a finite set $C_i$ of unitary representations of $G$, called the {\em cohomological representations of degree $i$}, such that for any uniform lattice $\Gamma \subset G$
\begin{equation} \label{matsushima_formula}
  b_i(\Gamma \bs X) = \sum_{\pi \in C_i} n(\pi, i) \cdot m(\pi, \Gamma) 
\end{equation}
where $n(\pi, i)=\dim \hom_K(\bigwedge^i\mathfrak p, \pi)$ and $\mathfrak p$ the representation of $K$ on the tangent space of $X$ at the identity coset. It is an integer depending only on $\pi$ and the degree $i$. It is known that cohomological representations can be discrete series only in degree $i = \dim(X)/2$, so \eqref{limmult} and \eqref{matsushima_formula} then imply the convergence in Theorem \ref{betti_numbers} (with $c(G)$ the sum of formal degrees of cohomological discrete series).

For a sequence of uniform lattices with bounded degree trace field, which are always uniformly discrete, the arguments in \cite[6.10]{7Sam} (see also \cite{DeWa79}) imply \eqref{limmult}. Using similar arguments, we give more precise estimates on the multiplicities using Theorem \ref{unbounded_degree} resulting in Theorem \ref{quant_betti_numbers}, which we do in Section \ref{s:quantbetti}. 

\medskip

On the other hand in \cite[Theorem 1.2]{7Sam} a stronger result is claimed, namely that for any open regular bounded Borel subset $S$ of the unitary dual of $G$ we have the limit
\begin{equation} \label{limmult_strong}
  \nu_{\Gamma_n}(S) := \frac 1{\vol(\Gamma_n \bs G)} \sum_{\pi \in S} m(\pi, \Gamma_n) \xrightarrow{n \to +\infty} \nu^G(S)
\end{equation}
for any uniformly discrete sequence $\Gamma_n$ such that $\Gamma_n \bs X$ is Benjamini--Schramm convergent  to $X$, where $\nu^G$ is the Plancherel measure of $G$.

The proof of \eqref{limmult_strong} rests on the two following claims:
\begin{enumerate}
\item The sequence of measures $\nu_{\Gamma_n}$ on the unitary dual converges weakly to the Plancherel measure;

\item The Plancherel measure on the unitary dual is characterised by its values on Fourier transforms of smooth, compactly supported functions on $G$ (see \cite[Proposition 6.4]{7Sam}). 
\end{enumerate}
The first claim is holds in our setting, as a formal consequence of Benjamini--Schramm convergence in the bounded-degree case and (with more work) as a consequence of Theorem \ref{unbounded_degree} in the other cases. The second claim is ``Sauvageot's density principle'' from \cite{Sauv96} and it is in this last reference that a gap has been found \cite[footnote 7 in section 24.2]{Nelson_Venkatesh}. Until this is fixed we cannot a priori claim that \eqref{limmult_strong} holds for all sequences of congruence latices.


\subsubsection{Limit multiplicities for non-cocompact lattices} \label{limmult_nonunif}

For non-uniform lattices it is much harder to prove limit multiplicities for Benjamini--Schramm convergent sequences because the Arthur trace formula, which extends the Selberg trace formula to this setting, includes terms which are not directly related to the geometry. 
For congruence subgroups in a fixed arithmetic lattice this was dealt with by Finis--Lapis in \cite{Finis_Lapid1,Finis_Lapid2} and for some lattices in groups of type $A_1$ this was dealt with by Matz \cite{Matz_SL2}. 


\subsection{Further questions}

To finish this introduction we state some open problems regarding locally symmetric spaces. 

\subsubsection{Optimal bounds for the volume of the thin part}

The best possible quantitative refinement of Theorem \ref{Main_BSconv} would be the following. 

\begin{conjecture}
  Let $G$ be a semisimple Lie group. There exists $\alpha, \beta > 0$ such that for any congruence arithmetic lattice $\Gamma$ in $G$ and $M = \Gamma \bs X$ the following holds
  \[
  \vol\left( M_{\le \beta\log\vol(M)} \right) \le \vol(M)^{1-\alpha}.
  \]
\end{conjecture}

A weaker version of this statement, where $\log\vol(M)$ is replaced by the degree of the trace field, was proved in \cite{Fraczyk} for torsion free lattices in $G = \PGL_2(\CC), \PGL_2(\RR)$. For congruence subgroups in a fixed arithmetic lattice a similar result is \cite[Theorem 1.12]{7Sam}. Our quantitative results Theorem \ref{statement_traceconv} and Theorem \ref{unbounded_degree} also go in this direction, though they are much weaker (in different ways each) than what would be needed to address this conjecture. 


\subsubsection{Non-congruence locally symmetric spaces}

As mentioned, in contrast to what happens for higher-rank spaces\footnote{Conditionally on the congruence subgroup property in some cases.}, for real and complex hyperbolic manifolds it is known that not all sequences BS-converge to hyperbolic space\footnote{For quaternionic and octonionic hyperbolic manifolds nothing is known about the possible BS-limits of finite volume locally symmetric spaces}. The following question of Shmuel Weinberger asks whether this is still true ``generically''. 

\begin{question}
  Let $G$ be $\mathrm{PO}(d, 1), d \ge 4$ or $\mathrm{PU}(d, 1), d \ge 2$ and $X$ its symmetric space. For $V > 0$ let $\mathcal M_G(V)$ be the set of $X$-orbifolds of volume at most $V$. Does
  \[
  \frac {\left|\left\{ M \in \mathcal M_G(V) : \frac{\vol(\Gamma_i \bs X)_{\le R}}{\vol (\Gamma_i \bs X)} > \eps \right\}\right|}{|\mathcal M_G(V)|}
  \xrightarrow{V\to+\infty} 0
  \]
  hold for all $\eps>0$ and $R>0$? 
\end{question}

For hyperbolic surfaces or 3--manifolds this question does not make sense because the ordering does not exist because of the failure of Wang's finiteness theorem. For surfaces one can formulate similar questions using random models and they have a positive answer for discrete models (see e.g. \cite[Appendix B]{survey}) or continuous ones (see \cite{Monk}). In these dimensions it also makes sense to restrict to arithmetic manifolds for which Wang finiteness holds \cite{Bor81}. For arithmetic  surfaces it seems very likely that the answer is positive, see \cite{MageePuder}. 


\subsection{Outline of the paper}

Section \ref{s:prelim} introduces the main objects we are interested in, the arithmetic lattices and their locally symmetric spaces, including the necessary background on Lie and algebraic groups. There we recall various standard results useful for our purposes which are not easy to find in the literature.

Sections \ref{s:BSconv}, \ref{s:galois} are preliminaries for the rest of the paper and many results therein are well known. Section \ref{s:BSconv} introduces Benjamini--Schramm convergence and gives a criterion for convergence, which we use right away to give a very short proof of Theorem \ref{Main_unbounded_degree}. In Section \ref{s:galois} we recall various definitions and facts about the Galois cohomology of semisimple algebraic groups and algebraic tori. 
Sections \ref{s:adelic} through \ref{s:small} contain the proof of Theorem \ref{Main_bounded_degree}; the first explains the proof using standard results on the arithmetic of algebraic groups together assuming certain estimates on orbital integrals the number of elements of small height. These estimates are proven in the sections \ref{s:orbint} and \ref{s:small}. 

The last section \ref{s:quantbetti} contains the (standard) arguments needed to deduce Theorem \ref{quant_betti_numbers} from Theorem \ref{unbounded_degree}. 


\subsection{Acknowledgments} 

S.~H. was supported by the Sloan Fellowship foundation. M. Fraczyk was partially supported by the Dioscuri programme initiated by the Max Planck Society, jointly managed with the National Science Centre in Poland, and mutually funded by Polish the Ministry of Education and Science and the German Federal Ministry of Education and Research. J.~R. was supported by grants AGIRA - ANR-16-CE40-0022 and ANR-20-CE40-0010 from the Agence Nationale de la Recherche.


\section{Algebraic groups and arithmetic lattices} \label{s:prelim}

\subsection{Number fields}

Throughout this paper we will use $k$ to denote a number field and $\ovl k$ an algebraic closure. The ring of integers of $k$ will be denoted by $\mo_k$. We will write $\Delta_k$ for the field discriminant of $k$. If $L/k$ is a finite extension, we write $\Delta_{L/k}$ for the relative discriminant, which is the ideal generated by the discriminants of all bases $a_1,\ldots, a_{[L:k]}\in \mo_L$, of $L$ over $k$. 

We will use $\val$ to denote the set of equivalence classes of valuations on $k$. For each $v \in \val$ we will use the same letter to denote the usual normalisation\footnote{absolute value in Archimedean places and $|\varphi|_{\frak p}=|\mo_k/\frak p|$ where $\varphi\in k_{\frak p}$ is a uniformizer.} of the equivalence class and $|\cdot|_v$ for the associated norm on $k$. The completion of $k$ at $v$ will be denoted by $k_v$. We will denote the set of Archimedean valuations by $\val_\infty$ and non-Archimedean ones by $\val_f$. If $v \in \val_f$, then $\mo_{k_v}$ will be the ring of integers in $k_v$. For non-Archimedean valuations we will use the same letter $\frak p$ for a prime ideal of $\mo_k$ and the corresponding $\frak p$-adic valuation.
If $\tau\in \Hom(k,\CC)$ we will write $|\cdot|_\tau$ for the valuation induced by the absolute value on $\CC$.

If $a \in k$ we denote by $N_{k/\QQ}(a)$ its absolute norm which is defined by: 
\[
N_{k/\QQ}(a) = \prod_{v \in \val_\infty} |a|_v = \prod_{v \in \val_f} |a|_v^{-1}.
\]
and for an ideal $I$, $N_{k/\QQ}(I) = |\mo_k/I|$. The two definitions coincide when $I=(a)$ is principal. The Weil  height of $a\in k$ is defined to be
\begin{equation} \label{defn_Weil_ht}
  h(a) = \frac{1}{[k:\QQ]}\sum_{v \in \val} \iota_v\log \max(1, |a|_v);  
\end{equation}
where $\iota_v=2$ if $k_v=\CC$ and $v=1$ otherwise. Note that if $a \in \mo_k$, then the sum can be restricted to $v \in \val_\infty$. This definition does not depend on the choice of the field $k$ containing $a$. The (logarithmic) {\em Mahler measure} is defined as 
\begin{equation} \label{defn_mahler_nb}
  m(a) = [\QQ(a):\QQ]h(a).  
\end{equation}

The adèle ring of $k$ will be denoted by $\Ade$, the infinite adèles $\prod_{v \in \val_\infty} k$ will be denoted by $k_\infty$ and the finite ones $\prod_{v \in \val_f}' k_v$ by $\Ade_f$.


\subsection{Algebraic groups and semi-simple elements}

Throughout this paper we will use $\G$ to denote a semisimple algebraic group defined over $k$; usually we will make this precise and add further hypotheses at every use of the notation. The adjoint representation of $\G$ on its $k$-Lie algebra will be denoted by $\ad$. 
The group $\G(\Ade)$ and $\G(\Ade_f)$ will denote respectively the groups $\Ade$ and $\Ade_f$ points of $\G$ (see \cite[pp. 248--249]{PlaRap}). 

Let $\gamma \in \G(k)$. We write $\G_\gamma$ for the centralizer of $\gamma$, which is an algebraic subgroup of $\G$ defined over $k$. If $\T$ is a torus in $\G$ we will denote by $X^*(\T)$ its group of characters defined over the algebraic closure $\ovl k$. It is a free abelian group of rank $\dim(\T)$ equipped with an action of the absolute Galois group of $k$ \cite{Ono1}. 

We will write $Z(\G),Z(G),Z(\Lambda)$ for the centre of an algebraic group $\G$, a Lie group $G$ or a discrete subgroup $\Lambda\subset G$.

We recall that a semisimple element $\gamma \in \G(k)$ is said to be {\em regular} if the connected component of $\G_\gamma$ is a torus. This is equivalent to the following. Let $\T \le \G$ be a maximal torus with $\gamma \in \T(k)$, and let $\Phi = \Phi(\G, \T) \subset X^*(\T)$ be the root system of $(\G, \T)$ \cite[2.1.10]{PlaRap}. Then, $\gamma$ is regular if and only if $\lambda(\gamma) \not= 1$ for all $\lambda \in \Phi$. An element $\gamma \in \G(k)$ is said to be 
{\em strongly regular} if it is regular and $\lambda(\gamma) \not= \lambda'(\gamma)$ for all $\lambda \not= \lambda' \in \Phi$.
The notion of strongly regular elements in not standard. We motivate its introduction by the following property.
\begin{proposition}\label{p-ConnectedCentralizer}
  Let $\gamma$ be a strongly regular element of $\G(k)$; then $\G_\gamma$ is a maximal torus of $\G$ which is split by the field $k[\lambda(\gamma)\,|\, \lambda\in \Phi]$, where $\Phi = \Phi(\G, \G_\gamma)$ is the root system of $\G$ relative to the torus $\G_\gamma$. 
\end{proposition}

\begin{proof}
  Since $\gamma$ is regular, the connected component $\T$ of $\G_\gamma$ is a maximal torus in $\G$ (\cite[2.11]{Steinberg}). We must prove that $\G_\gamma = \T$, i.e. that $\G_\gamma$ is connected. We have that $\G_\gamma \subset \mathbf J$ where $\mathbf J$ is the normaliser of $\T$ in $\G$, so that $W = \mathbf J(\ovl k)/\T(\ovl k)$ is the Weyl group of $\G$. It suffices to prove that $\gamma^w \not= \gamma$ for all $w \in W, w \not= 1$. The Weyl group $W$ acts faithfully on $\Phi$ and for any $\lambda \in \Phi$ we have $\lambda(\gamma^w) = \lambda^w(\gamma)$. Choosing a $\lambda$ with $\lambda^w \not= \lambda$ we get that $\lambda(\gamma^w) \not= \lambda(\gamma)$ by strong regularity of $\gamma$. Hence, $\gamma^w \not= \gamma$. This proves that $\G_\gamma$ is a torus.

  Now we prove that $L = k[\lambda(\gamma)\,|\, \lambda\in \Phi]$ splits $\T$. This amounts to showing that $\Gal(L)$ acts trivially on $X^*(\T)$, or on $\Phi$ \cite{Ono1}. Let $\sigma \in \Gal(L)$ and let $\lambda \in \Phi$; we have $\lambda^\sigma(\gamma) = \lambda(\gamma^\sigma) = \lambda(\gamma)$. For all $\lambda' \in \Phi \setminus \{\lambda\}$ we have $\lambda'(\gamma) \not= \lambda(\gamma)$. It follows that $\lambda = \lambda^\sigma$ must hold. Hence, $\sigma$ acts trivially on $\Phi$. 
\end{proof}


Let $\gamma \in \G(k)$ be a semisimple element. Choose a maximal torus $\T$ containing it and let $\Phi=\Phi(\G,\T)$ be the associated root system. We define the (logarithmic) Mahler measure $m(\gamma)$ by :
\begin{equation} \label{defn_mahler_mat}
  m(\gamma) = \sum_{\lambda \in \Phi} [k:\QQ]h(\lambda(\gamma)).
\end{equation}
where $h$ is the Weil height, defined by \eqref{defn_Weil_ht}. The {\em Weyl discriminant} of $\gamma$ is
\[
\Delta(\gamma) = \prod_{\lambda \in \Phi} (1 - \lambda(\gamma));
\]
if all $\lambda(\gamma) \in \mo_{\ovl k}$ then clearly $m(\Delta(\gamma)) \ll m(\gamma)$.

We say that an element $\gamma \in \G(k)$ has integral traces if $\tr \ad(\gamma^m)\in\mo_k$, for every $m \in \mathbb N$.

\begin{proposition} \label{bound_disc_splitting_field}
  For any $d \in \mathbb N$ there exists an increasing function $f_d : ]0, +\infty[ \to ]0, +\infty[$ such that the following holds. Let $k$ be a number field with $[k:\QQ] \le d$ and $\G$ a semisimple $k$-group. For any strongly regular $\gamma \in \G(k)$ which has integral traces in the adjoint representation, if $L$ is the minimal Galois splitting field for the centraliser $\G_\gamma$ and $m(\gamma) \leq R$ then
  \[
  N_{k/\QQ}(\Delta_{L/k}) \le f_d(R).
  \]
\end{proposition}

\begin{proof}
  Let $m = [L:k]$. By Proposition \ref{p-ConnectedCentralizer} the field $L$ is contained in $k[\lambda(\gamma)\,|\, \lambda\in \Phi]$. The degree $k[\lambda(\gamma)\,|\, \lambda\in \Phi]:k]$ is at most the degree of the characteristic polynomial of $\ad(\gamma)$, so $m$ is bounded in terms of $\G$. Since $\gamma$ has integral traces, $\lambda(\gamma)$ are algebraic integers.
  Let \[\Sigma:=\{\lambda(\gamma)^\ell| \lambda\in \Phi, \ell=0,1,\ldots, m-1\}\subset \mo_L.\] There exist $\theta_1, \ldots, \theta_m \in \Sigma$ which form a basis of $L$ over $k$. By definition of the relative discriminant \cite{LangANT}, $\Delta_{L/k}$ belongs to the principal ideal $I$ of $\mo_k$ generated by $\det((\theta_i^{\sigma_j})_{1\le i,j \le m})$ where $\Gal(L/k) = \{\sigma_1, \ldots, \sigma_m\}$. In particular, $N_{k/\QQ}(\Delta_{L/k}) \le N_{k/\QQ}(I)$. 

  Let $\tau_1,\ldots,\tau_{[k:\QQ]}$ be the set of all embeddings $k\to\CC$. We extend each $\tau_\ell$ to $L$ in an arbitrary way. We have 
  \[N_{k/\QQ}(\Delta_{L/k})=\prod_{\ell=1}^{[k:\QQ]}\left|\det(\theta_i^{\sigma_j})_{1\le i,j \le m})\right|_{\tau_\ell}=\prod_{\ell=1}^{[k:\QQ]}\left|\det((\theta_i^{\sigma_j\tau_\ell})_{1\le i,j \le m})\right|.\]
  Since $m(\gamma)$ is bounded by $R$, each absolute value $\left| \lambda(\gamma)^{\sigma_j\tau_\ell} \right|$ is bounded by $e^mR$ and consequently $|\theta_i^{\sigma_j\tau_\ell}|\le e^{m^2R}$. Therefore, the absolute value of the product is bounded by a function depending only on $R$ and $d$. 
\end{proof}


A $k$-torus $\T$ is said to be {\em $k$-anisotropic} if $X^*(\T)^{\Gal(k)} = 0$. In other words, it has no nontrivial character defined over $k$. This is equivalent to the quotient $\T(k) \bs \T(\Ade)$ being compact \cite[Theorem 4.11]{PlaRap}. A reductive $k$-group $\G$ is said to be $k$-anisotropic if all its $k$-tori are anisotropic. 


\subsection{Normalisation of measure}\label{sec:MeasureNorm}

\subsubsection{Lie groups}

We use the same normalisation of measure for reductive Lie groups as in \cite{PrasVol} and \cite{Ullmo_Yafaev}. Let $G$ be a real reductive group. The space of right $G$-invariant real differential $(\dim G)$--forms is a one dimensional vectors space. Any such non-zero form will $\omega$ gives rise to a right $G$-invariant measure $\mu_\omega$ on $G$ via the procedure described in \cite[p. 167]{PlaRap}. 

If $G$ is a semisimple real Lie group, there exists a unique compact semisimple real Lie group $\check{G}$, such that the complexifications $G_\CC$ and $\check{G}_\CC$ are isomorphic. Let $\check{\omega}$ be the unique, up to sign, real differential form on $\check{G}$, such that $\mu_{\check{\omega}}$ is the Haar probability measure. We fix a (complex) isomorphism $\iota: G_\CC\to \check{G}_\CC$ and define $\omega:=\iota^*(\check{\omega})$. The form $\omega$ turns out to be $G$--invariant, defined over $\RR$ and independent of the choice of $\iota$. We fix our choice of Haar measure on $G$ to be $\mu_{\omega}$.   

If $\T$ is a $\RR$-torus, the connected component of $\T(\RR)$ 
is isomorphic to a product $\SO(2)^a \times (\RR^\times)^b$. We normalise the Haar measure by taking the probability Haar measure on the $\SO(2)$ factors and $dt/|t|$ on the $\RR^\times$ factors. 

If $\T$ is a torus over a non-Archimedean local field $k_v$ then we normalise the Haar measure on $\T(k_v)$ so that the unique maximal compact subgroup has measure 1. If $\G$ is a semisimple group over $k_v$, then we will normalise the Haar measure on $\G(k_v)$ by choosing a compact-open subgroup of mass $1$ at each instance. 


\subsubsection{Homogeneous spaces}

We recall how the quotient measures on homogeneous spaces are defined: if $G$ is a locally compact group and $H$ a closed unimodular subgroup, then for any choices of two of Haar measures $dg$ on $G$, $dh$ on $H$ and a $G$-invariant measure $dx$ on $G/H$ the third is uniquely determined so that the equality
\[
\int_G f dg = \int_{G/H} \int_H f(xh) dh dx
\]
holds for any compactly supported function $f$ on $G$ (note that $\int_H f(xh) dh$ is a right-$H$-invariant function with compact support in the quotient). 


\subsection{Symmetric spaces}

If $G$ is a semisimple Lie group then the associated Riemannian symmetric space is $X = G/K$ where $K$ is a maximal compact subgroup, endowed with a $G$-invariant Riemannian metric. We use $d_X$ to denote the distance function on $X$. 
To normalize this metric we use the measure normalizations described above: if $G$ is (implicitly) endowed with a Haar measure we choose the Riemannian metric on $X$ so that its volume form co\"incides with the quotient measure where $\Vol(K) = 1$. 


\subsubsection{Semisimple isometries}

If $g \in G$ is semisimple and does not generate a bounded subgroup then we will call it {\em noncompact}. Such an element admits a unique totally geodesic subspace $\min(g)\subset X$ on which it acts with the minimal translation length. The distance $d_X(x, gx)$ is constant for $x \in \min(g)$; we call it the minimal translation length and denote by $\ell(g)$. The following proposition is well-know, see for instance \cite[proposition 2.2]{FHR_complexity}. 

\begin{proposition}\label{p-MahlerDisp}
  Let $X$ be the symmetric space of a semisimple Lie group $G$. There are constants $a_X, A_X > 0$ such that for any number field $k$, any semisimple $k$-group $\G$ such that the symmetric space of $\G(k_\infty)$ is isometric to $X$ (i.e. $\G(k_\infty)$ is isogeneous to $G$ times a compact group) and any semisimple $\gamma \in \G(k)$ which is non-compact in $\G(k_\infty)$ and has integral traces we have
  \begin{equation}\label{eq-MahlerDisp}
    a_Xm(\gamma) \le \ell(\gamma) \le A_X m(\gamma). 
  \end{equation}
\end{proposition}

A stronger condition than being noncompact is $\RR$-regularity. An element $g \in G$ is said to be {\em $\RR$-regular} if $\ad(g)$ has the maximal number of eigenvalues of absolute value $\not= 1$. 
These elements can also be characterised by the following result, which is \cite[Lemma 1.5]{Prasad_Raghunathan}. 

\begin{lemma}\label{lem:Rregular}
Let $\gamma$ be a semi-simple regular element of $G$. If $\gamma$ is $\RR$-regular then its centralizer $G_\gamma$ contains a maximal $\RR$-split torus of $G$.
\end{lemma}


\subsubsection{Injectivity radius and the thin part} \label{thinpart}

Let $d(g,h):=d_X(gK,hK)$. It is a bi--$K$--invariant, left $G$--invariant semi-metric on $G$. Let 
\[
\B(R) = \{ g\in G : d_X(K,gK) \leq R \}=\{g\in G : d(g,1)\leq R\}.
\]
One can think of $\B(R)$ as a ball in $G$, although it only corresponds to the semi-metric $d$ on $G$. Let $\Lambda\subset G$ be a discrete subgroup of $G$ and let $x = \Lambda g K\in \Lambda\bs X$ be a point in the locally symmetric space $M = \Lambda\bs X$. The {\em injectivity radius} of $M$ at $x$ is defined as 
\begin{align*}
  \InjRad_{\Lambda\bs X}(x) &= \sup\left\{ R\geq 0\,|\, d_X(gK, \gamma gK)\geq \frac R 2 \textrm{ for every } \gamma\in \Lambda \setminus Z(\Lambda) \right\} \\
  &=\sup\{R\geq 0\,|\, \Lambda^g\cap \B(2R)=Z(\Lambda)\}.
\end{align*}
If $M$ is a manifold (that is, if $\Lambda$ is torsion-free) then this is the same as the usual definition of injectivity radius. In general, this definition also takes into account the singular locus of the orbifold $M$. We also recall that the global injectivity radius is defined by $\InjRad_M := \inf_{x \in M} \InjRad_M(x)$. 

For $R>0$ the {\em $R$-thin} part of the locally symmetric space $M$ is the set 
\[
M_{\leq R} = \{x\in M \,|\, \InjRad_M(x)\leq R\}.
\]
The $R$-thick part $(\Lambda\bs X)_{\ge R}$ can be then defined as the closure of the complement of the $R$-thin part. 


\subsection{Arithmetic lattices} \label{prelim_lattices}

Let $G$ be a semisimple Lie group. A lattice $\Gamma\subset G$ is called arithmetic if there exists a $\QQ$-group $\mathbf H$ and a surjective morphism $\pi :\: \mathbf H(\RR) \to G$ with compact kernel such that $\Gamma$ and  $\pi (\mathbb H(\ZZ))$ are commensurable (i.e. have a common finite index subgroup). An {\em arithmetic locally symmetric space} is the quotient of a symmetric space $X = G/K$ by an arithmetic lattice in $G$. 

Alternatively, arithmetic lattices can be given a more constructive definition using the language of ad\`eles. This point of view is particularly well adapted to writing down trace formulas so that is what we are going to use. Let $k$ be a number field with ad\`ele ring $\Ade$ and let $\G$ be an adjoint semi-simple linear algebraic group over $k$. We will say that the couple $(\G,k)$ is {\em admissible} (for $G$) if $\G(k\otimes_\QQ \RR)\simeq G\times K$ with $K$ compact. Let $U$ be an open compact subgroup of $\G(\Ade_f)$. Then 
\[
\Gamma_U := \G(k)\cap (\G(k_\infty)\times U)
\]
is a lattice in $(\G(k_\infty)\times U)$ (see \cite[Chapter 5]{PlaRap}). By abuse of notation we will also write $\Gamma_U$ for the projection of $\Gamma_U$ to $G$, which is still a lattice since $U$ is compact. 

In general, an arithmetic lattice can be defined as any lattice which is commensurable to some $\Gamma_U$, and a {\em congruence arithmetic lattice} is one that contains some $\Gamma_U$. We note that the Borel density theorem and the commensurability invariance of the trace field (see below) imply that the commensurability class of $\Gamma_U$ is uniquely determined by the pair $(\G,k)$ and vice-versa. We recall that the lattice $\Gamma_U$ is uniform if and only if the $k$-group $\G$ is $k$-anisotropic \cite[Theorem 4.12]{PlaRap}.

The {\em (adjoint) trace field} of an arithmetic lattice $\Gamma$ is defined as the field $k$ where $(\G, k)$ is the unique admissible pair for which some or any $\Gamma_U$ is commensurable to $\Gamma$. Alternatively, the trace field is the field generated by the traces of elements of $\Gamma$ in the adjoint representation: 
\[
k = \QQ(\tr\ad(\gamma), \gamma\in \Gamma)
\]
That both definitions coincide is not obvious and follows from theorems of Vinberg. 
More precisely, let $\Gamma'$ be an arithmetic lattice in $\G(k)$ which is commensurable to $\Gamma$. Then we have obviously that $\ad(\Gamma') \subset \ad(\G(k))$ so $\QQ(\tr\ad(\Gamma') \subset k$; on the other hand we have $\QQ(\tr(\Gamma)) = \QQ(\tr(\Gamma'))$ as follows from \cite[Theorem 3]{Vinberg_trace}.  
For the reverse inclusion $k \subset \QQ(\tr\ad(\gamma), \gamma\in \Gamma)$ we use Vinberg's theorem: let $l$ be the trace field, by \cite[Theorem 1]{Vinberg_trace} $\ad(\Gamma)$ is contained in $\ad(\G(l_\infty))$. If $l \subsetneq k$, then the image of $\G(l_\infty)$ in $G$ is a proper subgroup of $G$ and this contradicts the fact that $\Gamma$ is Zariski-dense in the Zariski topology on $\G(k_\infty)$. 

The following proposition, the proof of which also uses Vinberg's results,  will also be of use later; in particular it implies that if $\Gamma$ is an arithmetic lattice then we can apply \eqref{eq-MahlerDisp} to any semisimple $\gamma \in \Gamma$. 

\begin{proposition}
  Let $\Gamma \subset G$ be an arithmetic lattice and let $\gamma \in \Gamma$ be a semisimple element. Let $T$ be a maximal torus containing $\gamma$. Let $\Phi$ be the roots associated to $T$ (over $\CC$). Then $\lambda(\gamma) \in \ovl\ZZ$ for all $\lambda \in \Phi$. 
\end{proposition}

\begin{proof}
  By definition of an arithmetic lattice, there exists a finite-index subgroup $\Gamma_1 \subset \Gamma$ and a $\mo_k$-lattice $L$ in the $k$-Lie algebra of $\G$ such that $\ad(\Gamma_1) L \subset L$. It follows that the ring generated by the traces $\tr(\ad(\gamma)), \gamma \in \Gamma_1$ is contained in $\mo_k$. By \cite[Theorem 3]{Vinberg_trace} so is the ring generated by the $\tr(\ad(\gamma)), \gamma \in \Gamma$. In particular all eigenvalues of elements in $\ad(\Gamma)$ are algebraic integers, which implies that the $\lambda(\gamma)$ must be an algebraic integer as well. 
\end{proof}

We use this result to give a lower bound on the systole of an arithmetic locally symmetric space using a lower bound on Mahler measure due to E.~Dobrowolski. Similar estimates have been obtained earlier in \cite{Belolipetsky2020ABF} and \cite{GelaVol}.

\begin{proposition} \label{dobrowolski_bound}
  There exists a constant $c$ depending only on $G$ such that if $\Gamma$ is an arithmetic lattice in $G$ with adjoint trace field $k$ and $d = [k:\QQ] \geq 2$, then for any semisimple non-compact $\gamma \in \Gamma$ we have
  \[
  \ell(\gamma) \ge \frac{c}{(\log d)^3}.
  \]
\end{proposition}

\begin{proof}
  The proof of the previous proposition shows that the minimal polynomial of $\ad(\gamma)$ belongs to $\mo_k$, so its roots $\lambda(\gamma)$ are algebraic integers of degree at most $d \cdot \dim(G)$. By \cite[Theorem 1]{Dobr79} it follows that $m(\lambda(\gamma)) \ge \frac 1{2(\log(d)+\log\dim(G))^3}$ for all $\lambda$ such that $\lambda(\gamma)$ is not a root of unity. Since $\gamma$ is non-compact there exists at least one such $\lambda$ and it follows that
  \[
  m(\gamma) \ge \frac 1{2(\log(d)+\log\dim(G))^3} \gg_G \frac 1{\log(d)^3}, 
  \]
  and the result follows by \eqref{eq-MahlerDisp}. 
\end{proof}


\subsection{Orbital integrals} \label{sec_orbint}

The following proposition is well-known (see for instance \cite[Theorem 27 on p. 215, Theorem 17 on p. 211]{Varadarajan}). 

\begin{proposition} \label{convergence_orbint}
  Let $\gamma$ be a semisimple element in a semisimple Lie group $G$. Then the map from $G/G_\gamma$ to $G$ defined by $g \mapsto g\gamma g^{-1}$ is proper. 
\end{proposition}

It follows immediately from this proposition that for any continuous compactly supported function $f$ the {\em orbital integral} $\mathcal O(\gamma, f)$ defined by
\[
\mathcal O(\gamma, f) = \int_{G_\gamma \bs G} f(x^{-1}\gamma x) dx
\]
is convergent.

Below we give a complete proof of Proposition \ref{convergence_orbint} in the case where $\gamma$ is regular non-compact, which uses several results that will be of use later, when we study further the orbital integral themselves both in Archimedean and non-Archimedean contexts.


\subsubsection{Commutators}

In this section $k_v$ is a local field, possibly non-Archimedean\footnote{The computation works over any field but we will use it only for these cases. }, and $\G$ is a $k_v$-group. 

\begin{lemma} \label{comm_ss_unip}
  Let $\T$ a maximal split torus in $\G$ and $\Phi = \Phi_{k_v}(\G, \T)$ the associated relative root system. Let $\N$ be a $k_v$-unipotent subgroup normalised by $\T$ and $\Phi_\N^+ \subset \Phi$ the subset of roots in $\N$. We denote by $t \mapsto n_\lambda(t)$ an isomorphism from $k_v$ to the root subgroup associated with $\lambda$. 

  We fix a linear order on $\Phi_\N^+$ so that any root cannot be decomposed as a sum of larger roots. Then there are constants $c_{\underline\nu}^\lambda \in k_v$ (depending only on the choice of isomorphisms $n_\lambda$) and $\eps_{\underline\nu, \nu}^\lambda \in \{0, 1\}$ (depending only on the chosen order on $\Phi$) for $\nu \subset \Phi_\N^+, \nu \in \underline\nu, \lambda \in \Phi_\N^+$ such that for all $\gamma \in \T(k_v)$ and $n=\prod_{\lambda\in\Phi_\N^+}n_\lambda(t_\lambda) \in \N(k_v)$ we have : 
  \[
  \gamma n^{-1} \gamma^{-1} n = \prod_{\lambda \in \Phi_\N^+} n_\lambda\left( (1-\lambda(\gamma))t_\lambda + \sum_{\underline\nu\subset \Phi_\N^+ : \sum_{\underline\nu} \nu = \lambda} c_{\underline\nu}^\lambda \prod_{\nu\in\underline\nu} \nu(\gamma)^{\eps_{\underline\nu, \nu}^\lambda} t_\nu \right). 
  \]
\end{lemma}

\begin{proof}
  Let $\lambda_1, \ldots, \lambda_m$ be the chosen linear ordering. 
  We have
  \[
  \gamma n^{-1} \gamma^{-1} n  = \prod_{i=m}^1 n_{\lambda_i}(-\lambda_i(\gamma)t_i) \prod_{i=1}^m n_{\lambda_i}(t_i).
  \]
  In the right-hand side, we move all the components of the first product one by one starting with $i=m$ to their rightful place in the second product. After this is finished we arrive at the final formula, where constants depend only on the commutators $[n_{\lambda_i}(1), n_{\lambda_j}(1)]$. 
\end{proof}


\subsubsection{Commutators at Archimedean places}

Here $k_v = \RR$ or $\CC$. If $H, L$ are Lie groups with Lie algebras $\mathfrak h, \mathfrak l$ then their tangent bundles admit a canonical isomorphism to $H \times \mathfrak h, L \times \mathfrak l$ respectively and so the Jacobian of a differentiable map $H \to L$ is just a map $H \to \hom(\mathfrak h, \mathfrak l)$. 

\begin{lemma} \label{constantjacobian}
  In the situation of Lemma \ref{comm_ss_unip} assume in addition that $\gamma$ acts non-trivially on every root subspace $\N_\lambda$, $\lambda \in \Phi_\N^+$. Then the map
  \[
  n \mapsto \gamma n^{-1}\gamma^{-1} n
  \]
  is a diffeomorphism of $\N$ onto itself and its Jacobian determinant is constant. 
\end{lemma}

\begin{proof}
  Since $\N(k_v)$ is simply connected, it suffices to prove the second assertion. It follows from Lemma \ref{comm_ss_unip} that the Jacobian matrix is upper triangular in the basis of $\mathfrak n$ coming from the chosen order on roots. Moreover the diagonal terms are constantly equal to  $1 - \lambda(\gamma)$, $\lambda \in \Phi_\N^+$ and the additional hypothesis on $\gamma$ ensures that these are all nonzero. 
\end{proof}


\subsubsection{Proof of Proposition \ref{convergence_orbint}}

Let $M := G_\gamma$, consider first the case where $M$ is a non-compact torus. In this case, $M$ is contained in a proper parabolic subgroup of $G$ so one can find a unipotent subgroup $N$ normalised by $M$ and maximal for this property. Then $MN$ is a parabolic subgroup of $G$ and we can choose a maximal compact subgroup $K$ of $G$ such that the Iwasawa decomposition $G = MNK$ holds (see \cite[Proposition 7.31]{Knapp_book_lie}). Since $K$ is compact, it suffices to prove that the map $N \to G$, $n \mapsto n^{-1} \gamma^{-1} n$ is proper. By Lemma \ref{constantjacobian} it is a diffeomorphism onto the closed subset $\gamma^{-1} N$ of $G$, so we are finished. 

In general $M$ might not be contained in a proper parabolic. However, if $M$ is non-compact then it contains a non-compact torus and we can always find a proper parabolic $P$ such that $P \cap M$ is cocompact in $M$ and a Levi component of $P$ is contained in $M$. Then considering the sequence of maps $K \times P \to K \times P/(P\cap M) \to G/M$ we need to prove that the map $N/(N\cap M) \to G$, $n \mapsto n\gamma n^{-1}$, is proper, where $N$ is the unipotent radical of $P$. This can be achieved by the same argument as above, using an appropriate modification of Lemma \ref{constantjacobian}. 

Finally, if $M$ is compact then $\gamma$ is compact and does not commute with any semi-simple non-compact element of $G$ and it is possible to prove the proposition using the $KAK$-decomposition.




\section{Benjamini--Schramm convergence} \label{s:BSconv}

In this section $G$ is an adjoint semisimple Lie group on which we fix a Haar measure (our statements will be independent of this choice). In the questions related to Benjamini-Schramm convergence we care only about the image of $G$ in the group of isometries of $X$, so passing to the adjoint group does not lose any information.

\subsection{Benjamini--Schramm convergence} \label{ss:BS_conv}

\subsubsection{Benjamini-Schramm convergence to the universal cover}

The Benjamini-Schramm topology is a topology defined on a subset of locally symmetric spaces, including all finite volume locally symmetric spaces and the space $X$ itself. In this paper we shall care only about the convergence of a sequence of finite volume spaces $(\Gamma_n\bs X)_{n\in\mathbb N}$ to $X$. We say the  $(\Gamma_n\bs X)_{n\in\mathbb N}$ {\em Benjamini-Schramm converges to $X$} if for every $R>0$ we have 
\[
\lim_{n\to\infty} \frac{\vol((\Gamma_n\bs X)_{\leq R})}{\vol(\Gamma_n\bs X)}=0.
\]
In more intuitive terms this means that the injectivity radius around a typical point of $\Gamma_n\bs X$ gets very large as $n$ goes to infinity. 


\subsubsection{The space of subgroups and invariant random subgroups}

Let $\sub(G)$ be the space of closed subgroups of $G$ equipped with the topology of Hausdorff convergence on compact sets. It is a compact $G$-space, often called the Chabauty space. The group $G$ acts on $\sub(G)$ by conjugation. By a slight abuse of notation, we can also speak about the injectivity radius of a subgroup of $G$:
\[
\inj(\Gamma):=\inf\{ d(\gamma K, K)\,|\, \gamma\in\Gamma\setminus Z(\Gamma)\}.
\]
Note that the injectivity radius of a non-discrete group is always $0$ but a discrete group might also have injectivity radius $0$, if it intersects $K$ non-trivially. An {\em invariant random subgroup} of $G$ is a $G$-invariant probability measure on $\sub(G).$ To every lattice $\Gamma\subset G$ we associate an invariant random subgroup $\mu_\Gamma$ supported on the conjugacy class of $\Gamma$ as follows : 
\[
\mu_\Gamma:= \frac{1}{\vol(\Gamma\bs G)}\int_{\Gamma\bs G} \delta_{\Gamma^g}dg.
\]
We can use the weak-$\ast$ topology on the space of probability measures on $\sub(G)$ to give an equivalent characterization of Benjamini-Schramm convergence as follows (see \cite[Proposition 3.6]{7Sam}).

\begin{lemma} \label{BS_topology}
  Assume that $G$ is an adjoint group. A sequence of locally symmetric spaces $\Gamma_n\bs X$ Benjamini-Schramm converges to $X$ if and only if the sequence of measures $\mu_{\Gamma_n}$ converges to $\delta_{\{1\}}$ in weak-$\ast$ topology.
\end{lemma}

Using this lemma, the question of whether a sequence of locally symmetric spaces Benjamini-Schramm converges to $X$ can be often answered by studying the possible limits of $\mu_{\Gamma_n}$. One result which often makes that possible is the Zariski density for invariant random subgroups; see the remark following \cite[Theorem 2.9]{7Sam} and \cite[Theorem 1.9]{Gelander_Levit_boreldensity}. 

\begin{theorem} \label{semisimple_IRS}
  Let $G$ be an adjoint semisimple Lie group without compact factors. Let $\mu$ be an ergodic invariant random subgroup of $G$. Then, there exists a normal subgroup $G_1$ of $G$ such that $\mu(\sub(G_1))=1$ and $\mu$-almost all subgroups are Zariski dense in $G_1$.  
\end{theorem}

In particular, if $G$ is simple and $\mu$ is an ergodic IRS, then either $\mu$ is trivial or it is Zariski dense almost surely.


\subsection{Main criterion for convergence}

Let $f$ be a smooth compactly supported function on $G$ (that is, $f \in C_c^\infty(G)$). Let $W$ be an open subset in $G$ which contains only semisimple elements and is invariant under conjugation. We define, for a lattice $\Gamma$ 
\begin{equation} \label{tf_conv}
  \tr R_\Gamma^W f := \sum_{[\gamma]_\Gamma \subset W} \vol(\Gamma_\gamma \bs G_\gamma) \mathcal O(\gamma, f),
\end{equation}
where the sum is over all conjugacy classes $[\gamma]_\Gamma$ of elements in $\Gamma\cap W$.
Let us address briefly issues of convergence in \eqref{tf_conv}. First we need that for all $\gamma \in W$ the covolume $\vol(\Gamma_\gamma \bs G_\gamma)$ is finite, equivalently that $\Gamma_\gamma \bs G_\gamma$ is compact. If $\Gamma$ is uniform, this is always the case. On the other hand, if it is non-uniform, then we must ensure that the co-volume is finite of all classes in $W$. The orbital integrals $\mathcal O(\gamma, f)$ are always finite (see Proposition \ref{convergence_orbint}). Finally, since the support of $f$ intersects only finitely many conjugacy classes of $\Gamma$ there is no issue with the convergence of the sum. 

We say that the subset $W$ is {\em sufficiently dense} in $G$, if for every discrete Zariski-dense subgroup $\Lambda$ in $G$, we have $\Lambda \cap W \not= \{1\}$. 

\begin{theorem} \label{general_criterion}
  Let $G$ be a semisimple adjoint Lie group. Let $W$ be a sufficiently dense subset of $G$. Let $(\Gamma_n)$ be a sequence of lattices in $G$. Assume either that they are all uniform or that $(\Gamma_n)_\gamma \bs G_\gamma$ is compact for all $n$ and $\gamma \in \Gamma_n \cap W$. 

  If $G$ is simple and the following condition holds:
  \begin{equation}\label{partialtraceconv}
    \forall f \in C_c^\infty(G) \: \lim_{n \to +\infty} \left(\frac 1{\vol(\Gamma_n\bs X)} \tr R_{\Gamma_n}^W f \right) = f(1_G), 
  \end{equation}
  then the sequence of locally symmetric spaces $\Gamma_n \bs X$ is Benjamini--Schramm convergent to $X$. 

  If $G$ is not necessarily simple then the same holds assuming that for every non-trivial proper normal subgroup $H \le G$ there exists a subset $F$ of $\sub_H$ such that $\Lambda \not\in F$ for all Zariski-dense subgroups $\Lambda \subset H$ and a neighbourhood $U_H$ of $\{\Lambda\in \sub_H :\:  \Lambda \not\in F\}$ in $\sub_G$ such that $g \Gamma_n g^{-1} \not\in U_H$ for all $n$ and $g \in G$.
\end{theorem}


\subsubsection{Proof when $G$ is simple}
  

Let $\mu_n$ be the invariant random subgroup of $G$ supported on the conjugacy class of $\Gamma_n$. We want to prove that any weak limit $\mu$ of a subsequence of $(\mu_n)$ is equal to the trivial IRS $\delta_{\{1\}}$. Since $G$ is simple, by Theorem \ref{semisimple_IRS} it suffices to prove that $\mu$ is supported on non-Zariski-dense subgroups to deduce that it must be trivial. 

To this end we choose a covering $W = \bigcup_{C \in \mathcal C} C$ of $W$ where $\mathcal C$ is countable and every $C \in \mathcal C$ is compact. We can do this since $\sub_G$ is metrizable \cite[Proposition 2]{dlHarpe1}. Let $Q_C = \{\Lambda : \Lambda \cap C \not= \emptyset \}$. This is a Chabauty-closed subset of $\sub_G$. If $\nu$ is a nontrivial IRS, then it is Zariski dense with positive probability so there exists $C \in \mathcal C$ such that $\nu(Q_C) > 0$. We want to prove the opposite for $\mu$, which amounts to the following : for every $C$, there exists a non-negative Borel function $F$ on $\sub_G$ which is positive on $Q_C$ and such that $\int_{\sub_G} F(\Lambda) d\mu(\Lambda) = 0$. 

Let us fix $C \in \mathcal C$ and prove this. There exists an open relatively compact subset $V$ with $C \subset V$ and  $\ovl V\subset W$. Choose any $f \in C^\infty(G)$ such that $f > 0$ on $C$ and $f = 0$ on $G \setminus V$ and define :
\[
F(\Lambda) = \begin{cases} \sum_{\lambda\in\Lambda} f(\lambda) &\text{if } \Lambda \text{ is discrete,} \\
  1 & \text{if } \Lambda \text{ is not discrete and intersects } C,\\ 
  0 & \text{otherwise.} \end{cases}
\] 
$F$ is lower semi-continuous on $\sub_G$, non-negative and positive on $Q_C$. On the other hand, we have :
\begin{align*}
  \int_{\sub_G} F(\Lambda) d\mu_n(\Lambda) &= \frac 1 {\vol(\Gamma_n \bs G)} \int_{G/ \Gamma_n} \sum_{\gamma \in g\Gamma_n g^{-1}} f(\gamma) dg \\
  &= \frac 1 {\vol(\Gamma_n \bs G)} \sum_{[\gamma]_{\Gamma_n} \subset W} \vol((\Gamma_n)_\gamma\bs G_\gamma)\mathcal O_\gamma(f). 
\end{align*}
By the so-called ``Portemanteau theorem'' \cite[Theorem 13.16]{klenke} the limit inferior of the left-hand side is larger or equal to $\int_{\sub_G} F(\Lambda) d\mu(\Lambda)$. Finally, the hypothesis \eqref{partialtraceconv} implies that the right-hand side converges to 0. It follows that
\[
\int_{\sub_G} F(\Lambda) d\mu(\Lambda) = 0
\]
which finishes the proof. 


\subsubsection{Proof in general}

Let $\mu_n$ be the invariant random sugroup supported on the $G$-orbit of $\Gamma_n$ and assume that $\mu$ is not supported on the trivial subgroup. Then by Theorem \ref{semisimple_IRS} there is a morphism $\pi :  G \to H$ with kernel $\ker(\pi) \not= \{1\}$ and
\[
\mu(\{ \Lambda \in \sub_{\ker(\pi)} : \Lambda \text{ is Zariski-dense in } \ker(\pi) \}) > 0. 
\]
If $\pi$ is trivial (that is $\ker(\pi) = G$) then the same arguments as in the proof in the simple case apply to give a contradiction. On the other hand, if $\ker(\pi)$ is a proper subgroup of $G$ then there exists a Zariski-dense subgroup $\Lambda \subset \ker(\pi)$ such that $\mu(C) > 0$ for every Chabauty neighbourhood $C$ of $\Lambda$. Since $\Lambda \not \in F$, we may choose $C$ such that $C \subset U_H$, and the hypothesis that $g\Gamma_n g^{-1} \not\in U_H$ for all $n$ and $g \in G$ then gives that $\mu_n(C) = 0$ for all $H$. A standard argument then shows that this contradicts the weak convergence of $\mu_n$ to $\mu$.

\subsection{Short proof of a non-effective result for unbounded degree} \label{s:cheat_mode}
This short subsection is dedicated to the proof of the following result, which is the ``second half'' of Theorem \ref{Main_BSconv}. It is strictly weaker than Theorem \ref{unbounded_degree} but since the argument is very short we feel it deserves to be included. 

\begin{theorem} \label{Main_unbounded_degree}
  Let $G$ be a noncompact semisimple Lie group and let $X$ be its symmetric space. Let $\Gamma_n$ in $G$ be a sequence of  arithmetic lattices in $G$ and assume that the degree of the trace field of $\Gamma_n$ over $\QQ$ goes to infinity with $n$. Then, \eqref{BSconv_univcover} holds for the sequence of locally symmetric spaces $\Gamma_n \bs X$. 
\end{theorem}

The main ingredient for the proof is the following results \cite[Theorem C]{FHR_complexity}. 

\begin{theorem} \label{AML}
  Let $G$ be a real semi-simple Lie group. There exists a constant $\varepsilon_G>0$ with the following property. Let $\Gamma\subset G$ be an arithmetic lattice with trace field $k$. Let $x\in X$. Then, the subgroup generated by the set 
  \[
  \{\gamma\in\Gamma\, |\, d(x,\gamma x)\leq \varepsilon_G [k:\QQ]\}
  \] 
  is virtually nilpotent.
\end{theorem}

\begin{proof}
We want to prove that \eqref{BSconv_univcover} holds for the locally symmetric spaces $\Gamma_n \bs X$. According to \ref{BS_topology}, this is equivalent to proving that the invariant random subgroups $\mu_{\Gamma_n}$ converge in weak topology to the Dirac mass $\delta_{\{1\}}$ supported on the trivial subgroup. Since the space of invariant random subgroups on $G$ is compact it suffices to prove that any accumulation point of a sequence $(\mu_{\Gamma_n})$ is equal to $\delta_{\{1\}}$. The idea of the proof is quite simple. We show that Theorem \ref{AML} force any limit to be supported on the set of virtually abelian subgroups of $G$. Then we use Zariski density for IRSs to deduce that such a limit must be concentrated on the trivial subgroup.

Suppose to the contrary that $\mu_{\Gamma_n}$ accumulates at a non-trivial IRS $\mu$. This means that there exists a closed subgroup $\Lambda \subset G$, $\Lambda \not=\{1\}$ in the support of $\mu$, so that $\mu(V) > 0$ for any Chabauty-neighbourhood $V$ of $\Lambda$. It follows that $V$ intersects the conjugacy class of $\Gamma_n$ for infinitely many $n$. In particular, $\Lambda$ is a limit of a sequence $g_n\Gamma_{k_n} g_n^{-1}$ in the Chabauty topology. We will deduce from this and the Arithmetic Margulis Lemma \ref{AML} that $\Lambda$ must be virtually abelian. This contradicts the Borel density for IRS (Theorem \ref{semisimple_IRS}) and suffices to establish the theorem. 

For this we need the following lemma \cite[Lemma 5.2]{FHR_complexity}. 

\begin{lemma} \label{bdd_index_abelian}
  For any $m$ there exists an $A$ with the following property. Let $\Delta$ be a finitely generated subgroup of $\mathrm{GL}_m(\CC)$ which contains an abelian subgroup of finite index 
  composed only of semisimple elements. Then, $\Delta$ has an abelian subgroup of index at most $A$. 
\end{lemma}

We go back to proving that $\Lambda$ must be virtually abelian. To simplify notation assume that the sequence $\Gamma_n$ actually converges to $\Lambda$ and that all $\Gamma_n$ are uniform. Almost all of them are uniform anyways, since any lattice with the trace field of degree $>\dim(G)$ must be uniform. Let $\eps_G$ be the constant given by Theorem \ref{AML} and recall that $k_{\Gamma_n}$ is the trace field of $\Gamma_n$. Fix $R > 0$. Since $[k_{\Gamma_n} : \QQ] \to +\infty$, we can assume that $[k_{\Gamma_n} : \QQ] \cdot \eps_G > R+10^{-2}$ for all $n$. Let
\[
\Lambda_{n, R} = \left\langle \Gamma_n \cap \B(R+10^{-2}) \right\rangle, \, \Lambda_R = \langle \Lambda \cap \B(R) \rangle.
\]
It follows from Theorem \ref{AML} and Lemma \ref{bdd_index_abelian} that there exists $A > 0$ and a sequence $H_{n, R} \subset \Lambda_{n, R}$ of abelian subgroups such that $[\Lambda_{n, R} : H_{n, R}] \le A$ for all $n$. The sequence $H_{n, R}$ converges to an abelian subgroup $H_R \subset \Lambda$; moreover (since any element of $\Lambda \cap \B(R)$ is a limit of a sequence  $\gamma_n \in \Gamma_n \cap \B(R+10^{-2})$) we see that $\Lambda_R \cap H_R$ has index at most $A$ in $\Lambda_R$.

In the previous paragraph we proved that for every $R>0$, the set $\mathcal A_R$ of abelian subgroups contained in $\Lambda_R$ with index at most $A$ is nonempty. Each set $\mathcal A_R$ is finite since $\Lambda_R$ is finitely generated and whenever $R < R'$ and $H \in \mathcal A_{R'}$ we have $H \cap \Lambda_R \in \mathcal A_R$, so we can pick a $H_1 \in \mathcal A_1$ such that for infinitely many $N \in \mathbb N$ there exists a $H_N \in \mathcal A_N$ with $H_1 \subset H_N$. We can then pick an integer $N_2 > 1$ and $H_{N_2} \in \mathcal A_{N_2}$ such that $H_1 \subset H_{N_2}$ and for infinitely many $N \in \mathbb N$ there exists a $H_N \in \mathcal A_N$ with $H_{N_2} \subset H_N$. Iterating this, we construct an increasing sequence $N_i \in \NN$ and $H_{N_i} \in \mathcal A_{N_i}$ which satisfy $H_{N_i} \subset H_{N_i+1}$. It follows that $H = \bigcup_{i \in \mathbb N} H_{N_i}$ is an abelian subgroup of $\Lambda = \bigcup_{i \in \mathbb N} \Lambda_{N_i}$ of index at most $A$. 
\end{proof}


\section{Galois cohomology} \label{s:galois}

In the remainder of the paper we will make intensive use of Galois cohomology of number fields and local fields. We regroup here various facts (some standard and some new) that we will use repeatedly in the later sections.

\subsection{Notation}

We let $\Gal(k)$ be the absolute Galois group of $\ovl k$ over $k$. If $L/k$ is a finite normal extension then $\Gal(L/k)$ denotes the corresponding finite Galois group.

\subsubsection*{Non-abelian cohomology}

If $\mathbf H$ is a $k$-group then $\mathbf H(\ovl k)$ (respectively $\mathbf H(L)$) is a $\Gal(k)$-group (respectively $\Gal(L/k)$-group) and we denote by $H^i(k, \mathbf H)$ (respectively $H^i(L/k, \mathbf H)$) the resulting non-abelian cohomology sets (for $i=0, 1$) which are defined in \cite[pp. 45--46]{Serre_galois}. In the case where $\mathbf H$ is abelian, the higher cohomology groups are defined as well. If $\mathbf H$ is non-abelian then $H^0(k, \mathbf H) = \mathbf{H}(k) = H^0(L/k, \mathbf H)$ and $H^1(k, \mathbf H)$ is the quotient set of
\[
Z^1(k, \mathbf H) = \{c :\: \Gal(k) \to \mathbf H(\ovl k),\, c(\sigma\tau) = c(\sigma)\cdot c(\tau)^\sigma \}
\]
by the equivalence relation $b \sim c$ if $c(\sigma) = v^{-1} b(\sigma) v^{\sigma}$ for some $v \in \mathbf H(\ovl k)$. The definition for $H^1(L/k, \mathbf H)$ is the same.

These sets are not groups unless $\mathbf H$ is abelian, however they have a distinguished point which is the class of the trivial morphism from $\Gal(k)$ to $\mathbf H(\ovl k)$. 


\subsubsection{Exact sequences}

If $\mathbf H$ is a $k$-subgroup of a $k$-group $\G$, then $\V = \G/\mathbf H$ is a $k$-variety. We can define $H^0(k, \V) := \V(k)$ and there is a long exact sequence
\begin{equation} \label{e-longcohomo}
  1 \to \mathbf H(k) \to \G(k) \to \V(k) \to H^1(k, \mathbf H) \to H^1(k, \G)
\end{equation}
(here exact means that the preimage of the distinguished point is the image of the preceding map); see \cite[Proposition 36]{Serre_galois}.

If $\mathbf H$ is normal in $\mathbf G$, then $\G/\mathbf H$ is a $k$-group and the sequence extends to a further term $H^1(k, \G/\mathbf H)$; if $\mathbf H$ is central in $\G$, then it extends to $H^2(k, \mathbf H)$ (see loc. cit., Propositions 38, 43).


\subsubsection*{Shafarevich--Tate groups}
The Shafarevich--Tate group $\Sha^1(\H)$ is the kernel of the map
\[
H^1(k, \H) \to \prod_{v \in \val} H^1(k_v, \H).
\]
It is always a finite group (see \cite[p. 284]{PlaRap}). 


\subsection{Galois cohomology of tori over local fields}

\subsubsection{Nakayama--Tate}


If $\T$ is a torus, then its character group $X^*(\T)$ is a $\Gal(k)$-module. If $L$ is a splitting field of $\T$, then the action of $\Gal(k)$ factors through $\Gal(L/k)$, so $X^*(\T)$ is also a $\Gal(L/k)$-module. The resulting cohomology groups are the same for both Galois groups, which follows from the inflation-restriction exact sequence. We have the following corollary of the local Nakayama--Tate theorem, see \cite[Theorem 6.2]{PlaRap}. 

\begin{theorem} \label{t-LocalNakayamaTate}
  Let $\mathbf T$ be an algebraic torus defined over a local field $k_v$ and split over a Galois extension $L_v/k_v$. Then there is an isomorphism $H^1(L_v/k_v,\mathbf T)\simeq  H^1(L_v/k_v, X^*(\mathbf T))$.
\end{theorem}


\subsubsection{Bounds on Shafarevich--Tate groups}

\begin{lemma}\label{l-TateShafarevichBound}
  Let $\mathbf T$ be an algebraic torus of dimension $r$ defined over $k$ and split over a Galois extension $L/k$. Then 
  \[
  |\Sha^1(\mathbf T)|\leq |H^2(L/k,X^*(\mathbf T))|\leq [L:k]^{r[L:k]^2}.
  \]
  If $\mathbf T$ is a maximal torus of a semisimple algebraic group $G$ then we can find $L$ with $[L:k]\leq (\dim \mathbf G-\rk \mathbf G)!$. In particular we have then that $|\Sha^1(\mathbf T)|=O_{\dim \mathbf G}(1)$.
\end{lemma}

\begin{proof}
  By \cite[Proof of Prop. 6.9 p. 306]{PlaRap} there is a surjective map $H^2(L/k,X^*(\mathbf T))\twoheadrightarrow \Sha^1(\mathbf T)$, which explains the first inequality. The second inequality is a standard bound on the size of cohomology groups of finite groups, it follows for instance from the fact that $H^2(L/k,X^*(\mathbf T))$ has exponent dividing $[L:k]$ \cite[Corollary 10.2, p.~84]{Brown_cohomology_book}, and there is a presentation complex for $\Gal(L/k)$ with at most $[L:k]^2$ two-cells (given by the multiplication table). 

  Now we prove the second statement. If $\mathbf T$ is a maximal torus in a semisimple algebraic group $\mathbf G$ we have an action on the root system $\pi:\Gal(k)\to \Aut(\Phi(\mathbf G,\mathbf T))$. Let us write $\rho:\Gal(k)\to \GL(X^*(\mathbf T))$ for the action of the Galois group on the character module of $\mathbf T$. The roots in $\Phi(\mathbf G,\mathbf T)$ span a subgroup of $X^*(\mathbf T)$ of full rank so $\ker \pi=\ker \rho$. In particular $[\Gal(k):\ker \rho]\leq |\Aut(\Phi)|\leq (\dim \mathbf G-\rk \mathbf G)!$ so the field ${\overline k}^{\ker\rho}$, which is the minimal splitting field of $\mathbf T$, is an extension of $k$ of degree at most $(\dim \mathbf G-\rk \mathbf G)!$. 
\end{proof}


\subsubsection{Bounds on $H^1$}

\begin{lemma} \label{bound_H1_torus}
  Let $k_v$ be a local field and let $\T$ be a $k_v$-torus. Then $\left| H^1(k_v, \T) \right|$ is bounded by a constant depending only on $\dim(\T)$. 
\end{lemma}

\begin{proof}
  Let $L_v$ be the minimal splitting field of $\T$. The Galois group $\Gal(L_v/k_v)$ acts on the character module $X^*(\T)$ by automorphisms. Write $J$ for the image of $\Gal(L_v/k_v)$ in $\GL(X^*(\T))$. Since $\GL_d(\ZZ)$ has only finitely many finite subgroups, there is a bound $A$ on $|J|$ depending only on $\dim(X^*(\T) \otimes\QQ) = \dim(\T)$. By Theorem \ref{t-LocalNakayamaTate}, we have $H^1(L_v/k_v,\T)\simeq H^1(L_v/k_v, X^*(\T))$. On the other hand we have
  \[
  |H^1(J, X^*(\T))|\leq |J|^{|J|\dim\T}\leq A^{A\dim\T},
  \]
  which finishes the proof. 
\end{proof}


\subsection{Compact cohomology classes in tori}

If $\mathbf T$ is defined over a local field $k_v$ and $L_v/k_v$ is a finite extension, we write $\mathbf{T}(L_v)^b$ for the unique maximal compact subgroup of $\mathbf T(L_v)$. We introduce the notion of compact cohomology classes which will be crucial in the proof of Proposition \ref{p-GConjugacyLocalUnr}.

\begin{definition}
  Let $\mathbf T$ be an algebraic torus defined over $k_v$, let $L_v$ be the minimal Galois extension of $k_v$ splitting $\T$. We say that a class $\alpha \in H^1(k_v,\mathbf T)$ is compact if it has a representative cocycle $c\in Z^1(L_v/k_v,\mathbf{T}(L_v)^b)$. 
\end{definition}

\begin{proposition}\label{p-CompactClassTriv}
  Let $\T$ be an algebraic torus defined over $k_v$ split by an unramified extension of $k_v$. Then every compact cohomology class in $H^1(k_v,\T)$ is trivial. 
\end{proposition}

\begin{proof}
  We compute all the Galois cohomology groups using the minimal splitting extension $L_v/k_v$ which is unramified by hypothesis. Let $\mathbf M$ be the multiplicative group over $k_v$, then there is a non-degererate pairing
  \[
  \langle \cdot, \cdot\rangle : H^1(L_v/k_v, \T) \times H^1(L_v/k_v, X^*(\T)) \to H^2(L_v/k_v, \mathbf M)
  \]
  induced by the non-degenerate pairing $\T \times X^*(\T) \to \mathbf M$ (the non-degeneracy of $\langle \cdot, \cdot\rangle$ is established in \cite[p. 303]{PlaRap}). A character $\chi \in X^*(\T)$ takes integer values on $\T(L_v)^b$, so if $\alpha \in H^1(L_v/k_v,\T)$ is compact, then the linear form $\langle \alpha, \cdot\rangle$ on $H^1(L_v/k_v, X^*(\T))$ takes values in the image of $H^2(L_v/k_v, \mo_{L_v}^\times)$. Since $L_v$ is unramified, the latter group is trivial by the Lemma \ref{l-Unr2Cohomo} below. We get that  $\langle \alpha, \cdot\rangle$ is trivial. Hence, $\alpha=0$, by the non-degeneracy of the pairing. 
\end{proof}

\begin{lemma}\label{l-Unr2Cohomo}
  The group $H^2(L_v/k_v,\mo_{L_v}^\times)$ vanishes for every unramified Galois extension $L_v/k_v$.
\end{lemma}

\begin{proof}
  Let $k_v^u$ be the unramified closure of $k_v$. Using the inflation-restriction long exact sequence (also known as Hochschild-Serre sequence) we get
  \[
  H^1(k_v^{u}/L_v,\mo_{k_v^u}^\times)^{\Gal(L_v/k_v)} \to H^2(L_v/k_v,\mo_{L_v}^\times) \to H^2(k_v^{u}/k_v,\mo_{k_v^u}^\times). 
  \]
  The first term vanishes by \cite[Theorem 6.8]{PlaRap} so the map  $H^2(L_v/k_v,\mo_{L_v}^\times)\to H^2(k_v^u/k_v,\mo_{k_v^u}^\times)$ is injective. To finish the proof it is enough to show that $H^2(k_v^u/k_v,\mo_{k_v^u}^\times)=0$. This follows from the argument in the proof of \cite[p. 299, Propostion 6.7]{PlaRap}, applied to the multiplicative torus $\bG_m$. 
\end{proof}


\section{Benjamini--Schramm convergence in bounded degree} \label{s:adelic}

In this section $G$ is a semisimple, simply connected Lie group. The reason to change the setting from the adjoint groups to simply connected ones is the Prasad volume formula, where many factors have been explicitly worked out only in the simply connected case.

\subsection{Finer description of arithmetic lattices} \label{defn_ppal_lattice}

\subsubsection{Bruhat--Tits buildings and parahoric subgroups} \label{BT_buildings}

We will describe some results of Bruhat--Tits theory that will be useful later. Let $\G$ be a simply connected, semisimple $k$-group and $v \in \val_f$. To the $v$-adic group $\G(k_v)$ there is an associated Euclidean building (see \cite[2.1]{Tits_Corvallis}) which we will denote by $X = X(\G, k_v)$. 

If $\T$ is a maximal $k_v$-split torus of $\G$, there is a unique apartment $A \subset X$ which is preserved by $\T(k_v)$. If $C$ is a chamber (i.e. a maximal simplex) of $A$, the associated Iwahori subgroup $I$ is the stabiliser of $C$ in $\G(k_v)$. Facets of $C$ are in 1-1 correspondence with the subsets of a basis of affine roots for $\G, \T$; the stabiliser of a facet (i.e. any compact subgroup containing the Iwahori subgroup) is called a {\em parahoric subgroup} (see \cite[0.5]{PrasVol}). 

Let $\Phi_a(\G, \T)$ be the affine root system over $k_v$ and let $\Delta_{a,v}$ the associated Dynkin diagram. The vertices of the diagram correspond to the top dimensional facets of the boundary of $C$. For $\Theta_v \subset \Delta_{a,v}$ we denote by $C_{\Theta_v}$ the associated facet of $C$, which is a simplex in $X$. By the previous paragraph, any parahoric subgroup $U_v$ of $\G(k_v)$ is conjugated to the stabiliser of a unique simplex $C_{\Theta_v}$, and we call $\Theta_v$ the {\em type} of $U_v$. 

The group $U_v$ of type $\Theta_v$ is said to be {\em special} if $\Delta_a \setminus \Theta_v$ is isomorphic to the Dynkin diagram of $\G$ over $k_v$. It is {\em hyperspecial} if and only if this remains true over any unramified finite extension of $k_v$ (see \cite[0.5, 0.6]{PrasVol}). 


\subsubsection{Principal and maximal arithmetic lattices}

A lattice of $\Gamma$ of $G$ is said to be \emph{maximal} if it is maximal with respect to inclusion among lattices.

Let $k$ be a number field and $\G$ a $k$-group such that there is a set of places $S_\infty \subset \val_\infty$ with $\prod_{v \in S_\infty} \G(k_v) \cong G$ and $\G(k_v)$ is compact for all $v \in \val_\infty \setminus S_\infty$. Note that then $\G$ must be semisimple and simply connected. By \ref{prelim_lattices}, any congruence subgroup $\Gamma_U \subset \G(k)$ is an arithmetic lattice in $G$. 

If $U_v, v \in \val_f$ are parahoric subgroups such that $U_v$ is hyperspecial for almost all $v$, and $U = \prod_{v \in \val_f} U_v$ then the lattice $\Gamma_U$ is called a {\em principal} arithmetic subgroup of $G$. The following proposition is contained in \cite[Proposition 1.4]{Borel_Prasad}.

\begin{proposition} \label{defn_congruence}
  A maximal arithmetic lattice in $G$ is the normalizer in $G$ of some principal arithmetic lattice. 
\end{proposition}


\subsubsection{Classification} \label{data_aritlattices}

The conjugacy class in $G$ of the principal arithmetic lattice $\Gamma_U$ is determined by the number field $k$, the isomorphism class of the $k$-group $\G$, and the local groups $U_v, v \in \val_f$ (equivalently the types $\Theta_v$). Regarding the group $\G$, we will only be interested in the following data\footnote{Note that while it does not completely characterise the $k$-isomorphism class of $\G$ it does so up to finite ambiguity. }:
\begin{enumerate}
\item \label{innerfield} The minimal field extension $\ell/k$ such that $\G/\ell$ is an inner form of its split form, except when $\G$ is of type ${}^6D_4$. In the latter case $\ell$ is defined as a subfield of this field with $\ell/k$ of degree 3 (see \cite[0.2]{PrasVol}). In any case $\ell$ is uniquely determined up to Galois conjugacy.

\item \label{notqs_places} The finite set of places $v \in \val_f$ such that $\G$ is not quasi-split over $k_v$. 
\end{enumerate}
Regarding the compact-open subgroup $U$ we will be using the following information:
\begin{enumerate}[resume]
\item \label{notspecial_places} The finite set of places $v \in \val_f$ such that $U_v$ is not special;

\item \label{nothyperspecial_places} and those where $\G$ splits over an unramified extension of $k_v$ but $U_v$ is not hyperspecial. 
\end{enumerate}
We will denote by $S_U$ the set of all "bad places" $v \in \val_f$ which fall in one of the sets in \ref{nothyperspecial_places}, \ref{notspecial_places} or \ref{notqs_places}. 


\subsection{Volume formula}

\subsubsection{Strong approximation} \label{iso_adelic_archimedean}

Since $G$ is noncompact and $\G$ is simply connected, the latter satisfies the strong approximation property with respect to infinite places \cite[Theorem 7.12]{PlaRap}. That is, $\G(k)$ is dense in $\G(\Ade_f)$. Let $dg_\infty$ be the standard Haar measure on $\G(k_\infty)$. Let  $dg_f$ the Haar measure on $\G(\Ade_f)$ normalised so that $\vol_{dg_f}(U) = 1$. Then, we have an isomorphism \[\Psi\colon (\G(k) \bs \G(\Ade) / U, dg_\infty\otimes dg_f)\to (\Gamma_U \bs G, dg_\infty)\] of measured spaces, where we define quotient measures by taking counting measure on the discrete subgroups $\G(k), \Gamma_U$ and Haar probability measure on the compact subgroup $U$. The map $\Psi$ is simply induced by the first projection $\G(\Ade) = \G(k_\infty) \times \G(\Ade_f) \to \G(k_\infty)$.


\subsubsection{Prasad's volume formula}

We will describe here part of the formula obtained by Gopal Prasad \cite{PrasVol} for the covolume $\vol(\Gamma_U \bs G)$ of a principal arithmetic lattice in $G$ with respect to the standard Haar measure on $G$ (normalized as in section \ref{sec:MeasureNorm}). The following proposition is a consequence of his results, tailored to our needs. In the sequel we use $C_G$ to denote a constant depending only on $G$. 

\begin{proposition} \label{volume_formula}
  Let $k$ be a number field, let $\G/k$ be a (simply connected) $k$-group such that $\G(k_\infty)$ is isomorphic to $G$ times a compact factor, and $U$ a compact-open subgroup which is a product of parahoric groups, so that $\Gamma_U$ is a principal arithmetic lattice. Then   
  \[
  \vol(\Gamma_U \bs G) \ge C_G |\Delta_k|^{\frac{\dim(\G)}2} \cdot N_{k/\QQ}(\Delta_{\ell/k}) \cdot \prod_{v \in S_U} q_v
  \]
  where $S_U, \ell$ are defined as in \ref{defn_ppal_lattice} above and $q_v$ is the cardinality of the residue field of $k_v$.
\end{proposition}

\begin{proof}
  This follows immediately from the exact formula in \cite[Theorem 3.7]{PrasVol} together with the following further facts from this reference: in the infinite product $\mathscr E$ all factors are $>1$ and they are $>q_v$ at places in $S_U$ (Proposition 2.10(iv) in loc. cit.), and the exponent $\mathfrak s(G)$ of $N_{k/\QQ}(\Delta_{\ell/k})^{1/2}$ is at least $2$ for all types (this can be checked case-by-case using the formulas given in loc. cit., 0.4). 
\end{proof}


\subsection{Index of maximal lattices}

To estimate the covolumes of maximal arithmetic lattices, we will use the following result. It follows from standard arguments from \cite{Margulis_Rohlfs} and \cite{Borel_Prasad}, as in \cite[Section 4]{Belolipetsky_Emery} which deals the case of even orthogonal groups. We essentially follow the argument of Belolipetsky--Emery in \cite[Section 4]{Belolipetsky_Emery}, cutting short where we do not need their degree of precision. 

\begin{proposition} \label{normaliser_index}
  Let $\Gamma_U$ be a principal arithmetic lattice in $G$ and $\Gamma = N_G(\Gamma_U)$ its normalizer. Let $\varepsilon>0$. Then 
  \[
  [\Gamma:\Gamma_U] \ll_\varepsilon C_G^{|S_U|} N_{k/\QQ}(\Delta_{\ell/k})^{\frac 1 2+\varepsilon} \Delta_k^{\frac 3 2+\varepsilon}.
  \]
\end{proposition}

\begin{proof}
 Let $h_\ell$ be the class number of $\ell$. We first prove that 
 \[[\Gamma:\Gamma_U] \le \begin{cases}C_G^{|S_U|} h_\ell&\text{ if }\G\text{ is not of type }{}^1D_{2m},\\ C_G^{|S_U|}h_k^2&\text{  otherwise.}\end{cases}\]
  Let $\ovl\G$ be the adjoint group of $\G$ and let $\pi$ be the morphism from $\G$ to $\ovl\G$. We have an exact sequence
  \[
  1 \to \Z_\G \to \G \to \ovl\G \to 1,
  \]
  which gives rise to the following exact sequence in Galois cohomology:
  \begin{equation} \label{galois_sequence_center}
    \G(k) \xrightarrow{\pi} \ovl\G(k) \xrightarrow{\delta} H^1(k, \Z_\G).
  \end{equation}
  As $\Gamma_U \subset \G(k)$ and the image of $\Gamma$ is contained in $\ovl\G(k)$, the map $\delta$ induces a map $\Gamma/\Gamma_U \to H^1(k, \Z_\G)$. By \cite[Proposition 2.6(i)]{Margulis_Rohlfs}, the following sequence is exact:
  \begin{equation} \label{rohlfs_ses}
    1 \to \Z_\G(\ovl k)/\Z_\G(k) \to \Gamma/\Gamma_U \xrightarrow{\delta} H^1(k, \Z_\G). 
  \end{equation}
  It remains to identify the image of $\Gamma/\Gamma_U$ in the (infinite) group $H^1(k, \Z_\G)$. For this, note that we can lift the morphism $\ovl\G(k) \to \prod_{v \in \val_f} \Aut(\Delta_{a, v})$ (induced by the conjugation action of $\ovl\G(k_v)$ on parahorics) to a map $\xi : H^1(k, \Z_\G) \to \prod_{v \in \val_f} \Aut(\Delta_{a, v})$ using the cohomology exact sequence \eqref{galois_sequence_center} since $\G(k)$ acts trivially on the local Dynkin diagrams $\Delta_{a, v}$. Then, $\delta(\Gamma/\Gamma_U)$ acts trivially at every finite place outside of $S_U$ (see \cite[4.2]{Belolipetsky_Emery}) and preserves the nontrivial types $\Theta_v$ at places in $S_U$. Together with \eqref{rohlfs_ses} this implies that
  \begin{equation} \label{first_estimate_index}
  \begin{split}
      |\Gamma/\Gamma_U| &\le |\Z_\G(\ovl k)| \cdot \prod_{v \in S_U} |\Aut(\Delta_{a, v})| \cdot |\{ c \in H^1(k, \Z_\G) :\: \forall v \in \val_f \setminus S_U,\, \xi(c)_v = \id\}| \\
      &\le |\Z_\G(\ovl k)| \cdot {C_G}^{|S_U|} \cdot |H_{\xi, S_U}|
    \end{split}
  \end{equation}
  since $|\Aut(\Delta)| \le C_G$ for any Dynkin diagram and a constant $C_G$ depending on $G$ and 
  \[
  H_{\xi, S_U} := \{ c \in H^1(k, \Z_\G) :\: \forall v \in \val_f \setminus S_U,\, \xi(c)_v = \id\}
  \]
  We estimate the size of this set following \cite{Belolipetsky_Emery}. We first need a proper description of $H^1(k, \Z_\G)$. According to the table on \cite[p.~332]{PlaRap}, the center $\Z_\G$ is of the form $\mu_b$ or $\res_{\ell/k}^{(1)}(\mu_b)$, where $\ell$ is the field defined in \ref{innerfield} above, $b$ is a positive integer depending only on $G$ and $\mu_b$ is the algebraic group defining $b$-root of unity. The only exception is in the case where $\G$ is inner of type $D_{2m}$. Then  $\Z_\G=\mu_2 \times \mu_2$. We will now restrict to the former two cases. There is always a map from $H^1(k, \Z_\G)$ to $\ell^\times/(\ell^\times)^b$. It is an isomorphism in case $\ell=k$ and otherwise comes from a long exact sequence in Galois cohomology. The size of its kernel is bounded by $b$, see (17) in \cite{Belolipetsky_Emery}. Let
  \begin{align*}
    \ell_{b, S_U}^\times &= \{ x \in \ell:\: \forall v \in \val_f \setminus S_U,\, v(x) = 0 \pmod b \} \\
    \ell_b^\times &= \{ x \in \ell:\: \forall v \in \val_f,\, v(x) = 0 \pmod b \}
  \end{align*}
  Then $\ell_{b, S_U}^\times \supset \ell_b^\times \supset (\ell^\times)^b$ and by \cite[4.9]{Belolipetsky_Emery} (see also \cite[Section 2]{Borel_Prasad}) we have that $H_{\xi, S_U} \subset \ell_{b, S_U}^\times/(\ell^\times)^b$. Hence 
  \[
  |H_{\xi, S_U}| \le |\ell_{b, S_U}^\times/(\ell^\times)^b|. 
  \]
  By the weak approximation property for the multiplicative group, $\ell_{b, S_U}^\times/\ell_b^\times \cong (\ZZ/b\ZZ)^{|S_U|}$. We have
  \[
  |\ell_{b, S_U}^\times/(\ell^\times)^b| = b^{|S_U|} \cdot |\ell_b^\times/(\ell^\times)^b|
  \]
  and \cite[Proposition 0.12]{Borel_Prasad}) gives the estimate
  \[
  |\ell_b^\times/(\ell^\times)^b| \le b^{[\ell:\QQ]} h_\ell
  \]
  so that 
  \[
  |H_{\xi, S_U}| \le b \cdot b^{|S_U|} b^{[\ell:\QQ]} h_\ell
  \]
  which (as $[\ell:\QQ]$ is assumed to be bounded), together with \eqref{first_estimate_index} finishes the proof when $\G$ is not of type ${}^1D_{2m}$.

  \medskip
  
  It remains to deal with $\G$ of type ${}^1D_{2m}$. In this case $\ell=k$ and $\Z_\G=\mu_2 \times \mu_2$. The argument above yields $|\Gamma/\Gamma_U| \le C_G^{|S_U|} \cdot |h_k|^2$.
  
  We can now finish the proof of the proposition. By the Brauer--Siegel Theorem, $h_\ell\ll_\varepsilon \Delta_\ell^{1/2+\varepsilon}$,  so 
  \[h_\ell\ll_\varepsilon \Delta_\ell^{\frac{1}{2}+\varepsilon} = N_{k/\QQ}(\Delta_{\ell/k})^{\frac{1}{2}+\varepsilon} \Delta_k^{\frac{[\ell:k]}{2}+\varepsilon}\leq N_{k/\QQ}(\Delta_{\ell/k})^{\frac{1}{2}+\varepsilon} \Delta_k^{\frac{3}{2}+\varepsilon},\] because $[\ell:k] \le 3$. In the case of type ${}^1D_{2m}$ we use Brauer-Siegel estimate to get $h_k^2\ll_\varepsilon \Delta_k^{1+\varepsilon}$.
  Proposition follows. 
\end{proof}

Note that the image of $\ell_b^\times$ in $h_\ell$ 
could be much smaller than the whole class group. This is where, for the fields of large discriminant, the main loss occurs. Conjecturally, the image is of size $\ll \Delta_\ell^\eps$ for any $\eps>0$ but even getting an exponent $<1/2$ seems hard and is not known in general (see \cite{BSTTTZ} for more context and a result for $b=2$). 


\subsection{Convergence} \label{proof_bounded_degree}

In this subsection we prove Theorem \ref{Main_bounded_degree}, admitting estimates which we will prove in the next two sections.


\subsubsection{Adelic trace formula}

For a function $f$ on $G$ and a principal arithmetic lattice $\Gamma_U$, we put $f_\Ade = f \otimes 1_{U_\infty} \otimes 1_U$, where $U_\infty$ is the product of $\G(k_v)$ for $v \in S_\infty$ with $\G(k_v)$ compact.  For a rational conjugacy class $[\gamma]_{\G(k)}$, we write $[\gamma]_{\G(k)}\subset W$ if the Archimedean part satisfies $[\gamma]_{\G(k_\infty)}\subset W\times U_\infty\subset \G(k_\infty).$

\begin{proposition} \label{adelic_reformulation}
  Let $f, W$ be as in Theorem \ref{general_criterion}. 
  We have
  \begin{equation} \label{trace}
  \tr R_{\Gamma_U}^W f = \sum_{[\gamma]_{\G(k)} \subset W} \vol(\G_\gamma(k) \bs \G_\gamma(\Ade)) \int_{\G(\Ade)/\G_\gamma(\Ade)} f_\Ade(x\gamma x^{-1}) dx
  \end{equation}
\end{proposition}

\begin{proof}
  Let $\Psi$ be the isomorphism from $\G(k) \bs \G(\Ade) / U$ to $\Gamma_U \bs G$ described in \ref{iso_adelic_archimedean}. Let $F$ (resp. $F_\Ade$) be the $\Gamma_U$-invariant (resp. $\G(k)$-invariant) function defined by $F(x) = \sum_{[\gamma]_{\Gamma_U} \subset W} f(x\gamma x^{-1})$ (resp. $F(x) = \sum_{[\gamma]_{\G(k)} \subset W} f_\Ade(x\gamma x^{-1})$. Then we have that $F_\Ade = F \circ \Psi$ and since local measure are normalized so that $U_v$ has volume 1 it follows that $\int_{\Gamma \bs G} F(x) dx = \int_{\G(k) \bs \G(\Ade)} F_\Ade(x) dx$ (see also the proof of Lemma 4.3 in \cite{Fraczyk}). On the other hand the usual unfolding trick (see for instance \cite{Clay03}) shows that $\int_{\Gamma \bs G} F(x) dx$ (resp. $\int_{\G(k) \bs \G(\Ade)} F_\Ade(x) dx$)is equal to the term on the  left-hand (resp right-hand) side of the equality \eqref{trace}, which proves the latter. 
\end{proof}


\subsubsection{The case of principal arithmetic lattices}

In \ref{endproof} below (and \ref{nonuniform} for the non-compact case which requires some additional arguments) we will deduce Theorem \ref{Main_bounded_degree} from the following proposition. 

\begin{proposition} \label{principal_case}
  Let $W$ be a subset of the set of strongly regular, $\RR$-regular semisimple elements in $G$. Let $\Gamma_U$ be a principal arithmetic lattice and let $k, S_U$ be defined as in \ref{data_aritlattices}. Then, for any $f \in C_c^\infty(G)$ we have 
  \[
  \tr R_{\Gamma_U}^W f \ll_{f, d} \Delta_k^{\frac{\dim(\G)}2 - \delta}
  \]
  for any  $\delta < \tfrac{\dim(\G) - \dim(\T)}2$, where $\T$ is a maximal $k$-torus in $\G$ and $d=[k:\QQ]$. 
\end{proposition}

In view of Proposition \ref{adelic_reformulation} above, Proposition \ref{principal_case} is an immediate consequence of the following theorem. The theorem itself is the main technical contribution in this part of our work and the proof depends on estimates that we will prove in the next two sections. 

\begin{theorem} \label{statement_traceconv}
  Let $G$ be an adjoint Lie group and $f$ a continous, compactly supported function on $G$. Let $W$ be a subset of the set of strongly regular, $\RR$-regular semisimple elements in $G$. For fixed $d \in \NN$ and all number fields with $[k:\QQ] = d$ and $\G(k_\infty) \simeq G$ we have
  \[
  \sum_{[\gamma]_{\G(k)} \subset W} \vol(\G_\gamma(k) \bs \G_\gamma(\Ade)) \int_{\G_\gamma(\Ade)\bs \G(\Ade)} f_\Ade(x^{-1}\gamma x) dx \ll_{f,d} \Delta_k^{\frac{\dim(\G)}2 - \delta}
  \]
  for any  $\delta < \tfrac{\dim(\G) - \dim(\T)}2$, where $\T$ is a maximal $k$-torus in $\G$. 
\end{theorem}

\begin{proof}
  Let
  \[
  \mathcal O(\gamma, f_\Ade) = \int_{\G_\gamma(\Ade)/\G(\Ade)} f(x\gamma x^{-1}) dx
  \]
  We explain here how to estimate the sum
  \[
  \sum_{[\gamma]_{\G(k)} \subset W} \vol(\G_\gamma(k) \bs \G_\gamma(\Ade)) \cdot \mathcal O(\gamma, f_\Ade) , 
  \]
  leaving most technical details to the next two sections. The support of $f_\Ade$ is compact in $\G(\Ade)$ so this is a finite sum. In fact, if we assume that $f$ has support in a ball $B_G(1, R)$, for some radius $R$, the we can restrict the sum to the conjugacy classes $[\gamma] \subset \G(k)$ which satisfy $m(\gamma) \le R$ 
  and $\gamma \in U$.

  For such $\gamma$, we prove in Theorem \ref{global_orbint} that for the Haar measure normalised so that $\vol(U)=1$ we have 
  \begin{equation} \label{orbint_statement}
    \mathcal O(\gamma, f_\Ade) \ll_{f, [k:\QQ]} 1.
  \end{equation}
  To estimate the co-volumes of the centralizers we follow the proof of \cite[Lemma 6.4]{Fraczyk}, where it is deduced from results of Ullmo and Yafaev \cite{Ullmo_Yafaev}. For any $k$-torus $\T$ we have
  \[
  \vol(\T(k) \bs \T(\Ade)) \ll |\Delta_k|^{\frac r 2} \cdot N_{k/\QQ}(\Delta_{L/k})^{\frac 1 2} \cdot L(1, \chi_\T)
  \]
   where $r = \dim(\T)$ and  $L(\cdot, \chi_\T)$ is the Artin $L$-function associated with the character $\chi_\T$ of the representation of $\Gal(k)$ on $X^*(\T)$. The $L$-function is holomorphic at $1$ since under our assumptions the torus $\T$ is anisotropic. Under our assumptions on $\gamma$, it follows from Proposition \ref{bound_disc_splitting_field} that $N_{k/\QQ}(\Delta_{L/k}) = O(1)$. For the term $L(1\, \chi_\T)$, by \cite[Proposition 2.1]{Ullmo_Yafaev} we have that for any $\eps > 0$
  \[
  L(1, \chi_\T) \ll_{\eps, [L:\QQ]} |\Delta_L|^\eps
  \]
  and the right-hand side is equal to $N_{k/\QQ}(\Delta_{L/k})^\eps |\Delta_k|^{\eps\cdot[L:k]}$. Consequently
  \begin{equation} \label{vol_tori}
    \vol(\T(k) \bs \T(\Ade)) \ll_{f, [k:\QQ], \eps}  |\Delta_k|^{\frac {r+\eps} 2} 
  \end{equation}
  for any $\eps > 0$. 

  Let $\mathcal C(R, U)$ be the set of conjugacy classes $[\gamma] \subset \G(k)$ which satisfy $m(\gamma) \le R$ and $\gamma \in U$. It follows from \eqref{orbint_statement}, \eqref{vol_tori} that
  \[
  \sum_{[\gamma]_{\G(k)} \subset W} \vol(\G_\gamma(k) \bs \G_\gamma(\Ade)) \mathcal O(\gamma, f_\Ade) \ll_{f, [k:\QQ], \eps} |\mathcal C(R, U)| \cdot |\Delta_k|^{\frac {r+\eps} 2}. 
  \]
  In Theorem \ref{t-SmallHeigtCount} we prove that $|\mathcal C(R, U)| \ll_{R, [k:\QQ]} 1 $. Writing $\delta = \tfrac{\dim(\G) - r - \eps} 2$ we get that 
  \[
  \sum_{[\gamma]_{\G(k)} \subset W} \vol(\G_\gamma(k) \bs \G_\gamma(\Ade)) \mathcal O(\gamma, f_\Ade) \ll_{f, [k:\QQ], \eps} |\Delta_k|^{\frac{\dim(\G)}2 - \delta}.
  \]
  Since $\eps >0$ can be chose arbitrarily small, $\delta$ can be chosen arbitrarily close to $\tfrac{\dim(\G) - r} 2$. The proof is complete. 
\end{proof}


\subsubsection{Preliminary lemmas}

Various estimates on the volumes of centralizers or the orbital integrals become easier if we restrict to conjugacy classes with additional regularity properties. This is why we introduced the set $W$. To deal with maximal lattices, we will need the notion of strong $m$--regularity ($m \in \NN$), which we will not use anywhere else in the paper. If $\G$ is a semisimple group, $\T$ a maximal torus in $G$ an element $g \in T$ is {\em strongly $m-$regular} if $\xi(g)\neq 1$ for all non-trivial characters $\xi\in X^*(\T)$ such that $\xi$ is a product of at most $2m$ roots. We note that strongly $m$-regular elements for any $m$ are a Zariski-dense subset. 

\begin{lemma} \label{suff_dense}
  Let $G$ be a semisimple Lie group and $m \in \NN$. The set of strongly $m$--regular, $\RR$-regular elements is sufficiently dense in $G$. 
\end{lemma}

\begin{proof}
  In \cite{PrasadR} Prasad gives a proof that any Zariski-dense subgroup of a semisimple Lie group contains an $\RR$-regular element. His argument actually implies the stronger statement above, as we now explain. At the beginning of the proof Prasad introduces a numerical invariant $m(\gamma)$ for $\gamma \in G$, which equals $1$ if and only if $\gamma$ is $\RR$-regular. In Lemma B he proves that for any Zariski-dense subgroup $\Lambda$ in $G$ the minimum $\min_{\gamma \in \Lambda} m(\gamma)$ is reached on the subset of regular elements; the rest of the proof establishes that $\min_{\gamma \in \Lambda} m(\gamma) = 1$. On the other hand, the proof of Lemma B immediately implies that the minimum is reached on the intersection of $\Lambda$ with any Zariski-open subset of $G$. In particular, it works for strongly $m$--regular elements. 
\end{proof}

In order to check the additional condition appearing in Theorem \ref{general_criterion} for non-simple groups we will require the following lemma. 

\begin{lemma} \label{away_from_diagonal}
  Let $L, H$ be two adjoint semisimple groups and $G = L \times H$. There exists a neighbourhood $V$ of $\sub_H \setminus \{\{1_H\}\}$ which does not contain any irreducible arithmetic lattice of $G$. 
\end{lemma}

\begin{proof}
  Let $\|\cdot\|_H, \|\cdot\|_L$ be some norms on $H$ and $L$. Let $\pi_L, \pi_H$ be the projections from $G$ to $L, H$ respectively. We define the following open subsets of $\sub_G$ :
  \[
  V_{R, \eps} = \{\Lambda \in \sub_G :\: \exists g \in \Lambda \setminus \{1\} :\: \|\pi_L(g) - 1\|_L < \eps,\, \|\pi_H(g)\| < R \}.
  \]
  For any $R_\eps > 0$ such that $\lim_{\eps \to 0} R_\eps = +\infty$ we have that $\bigcup_{R > 0} V_{R_\eps, \eps}$ is a neighbourhood in $\sub_G$ of $\sub_H \setminus \{\{1_H\}\}$. We want to prove that for any $\eps > 0$, there exists $R_\eps > 0$ such that for every irreducible arithmetic lattice $\Gamma$ in $G$ the following is true. If $\gamma \in \Gamma \setminus \{1\}$ satisfies $\|\pi_L(g) - 1\|_L < \eps$, then $\|\pi_H(g)\| \ge R_\eps$, and $\lim_{\eps \to 0} R_\eps = +\infty$. By construction, the set $V:=\bigcup_{R > 0} V_{R_\eps, \eps}$ will not contain any irreducible arithmetic lattice of $G$.

  Since $\Gamma$ is irreducible arithmetic and $G$ is adjoint, there exists a number field $k$, a $k$-group $\G$ and a partition $\val_\infty = S_1 \cup S_2 \cup S_3$ such that
  \[
  H = \G(k_{S_1}), \, L = \G(k_{S_2}), 
  \]
  $\G(k_{S_3})$ is compact and $\Gamma$ is an arithmetic subgroup of $\G(k)$. In particular, the eigenvalues $\lambda_1, \ldots, \lambda_m$ of $\ad(\gamma)$ (where $\ad$ is the adjoint representation of $\G$ on its $k$-Lie algebra) are algebraic integers. For each $i$
  \[
  \prod_{v \in S_2} | \lambda_i - 1 |_v \ll \eps
  \]
  (with constants depending only on the choice of norm on $L$) and 
  \[
  \prod_{v \in S_3} | \lambda_i - 1 |_v < 2^{|S_3|} < 2^d. 
  \]
  On the other hand, since $\lambda_i$ are algenraic integers and they are not all equal to $1$, there must be an $1 \le i \le m$ such that 
  \[
  \prod_{v \in \val_\infty} |\lambda_i - 1| \ge 1.
  \]
  It follows that
  \[
  \prod_{v \in S_1} |\lambda_i - 1| \gg 2^{-d} \eps^{-1}.
  \]
  This finishes the proof, since $\| \pi_H(\gamma) - 1 \| \gg \prod_{v \in S_1} |\lambda_i - 1|$ (with a constant depending only on the choice of norm on $H$) so can choose $R_\eps \gg \eps^{-1}$. 
\end{proof}


\subsubsection{The proof of Theorem \ref{Main_bounded_degree} for uniform lattices} \label{endproof}

Let $m$ be an integer (to be fixed later, it will depend only on $G, d$) and let $W$ be the set of $\RR$-regular, strongly $m$-regular elements of $G$. It follows from Lemmas \ref{suff_dense}, \ref{away_from_diagonal} and Theorem \ref{general_criterion} (with $F = \{1\}$ and $U$ given by \ref{away_from_diagonal}) that we need only to prove that \eqref{partialtraceconv} holds for any sequence $\Gamma_n$ of maximal arithmetic lattices whose trace fields have degree $d$.

We start with the following lemma.

\begin{lemma} \label{last_bound1}
  The exponent of the finite group $\Gamma/\Gamma_U$ is bounded by a constant depending only on $G$ and the degree $[k:\QQ]$. 
\end{lemma}

\begin{proof}
  By \cite[Proposition 2.6(i)]{Margulis_Rohlfs} there is an exact sequence
  \[
  0 \to \mathbf Z_{\G}(\ovl k) / \mathbf Z_{\G}(k) \to \Gamma_n/\Gamma_{U_n} \to H^1(k, Z_{\G}).
  \]
   The group $H^1(k, Z_{\G})$ has finite exponent depending only on the absolute type of $\G$ so the same must hold for $\Gamma/\Gamma_U$. 
\end{proof}

In the sequel we fix $m$, taking it to be an upper bound for $\Gamma/\Gamma_U$ given by this lemma. We will prove that for this choice of $m$, 
\begin{equation} \label{nbcc_for_maximal}
  \tr R_\Gamma^{W} 1_{B(R)} = o(\vol(\Gamma \bs G))
\end{equation}
for a uniform maximal arithmetic lattice $\Gamma$ in $G$ with degree of trace field  $[k_\Gamma:\QQ] = d$. This implies \eqref{partialtraceconv} since we can approximate any compactly supported function by sums of translates of functions of the form $1_{B(R)}$ and in view of \cite[Theorem 8.36]{Knapp_book_lie} it implies directly the geometric statement \eqref{BSconv_univcover}). 

Let $\Gamma$ be such a lattice. By the description in Proposition \ref{defn_congruence}, there exists a principal arithmetic lattice $\Gamma_U$ in $G$, such that $\Gamma = N_G(\Gamma_U)$. We recall Prasad's volume formula and the index estimates we discussed above :
\begin{align}
  &\vol(\Gamma_U \bs G) \ge C_G^{-1} \Delta_k^{\frac{\dim(\G)}2} \cdot N_{k/\QQ}(\Delta_{\ell/k}) \cdot \prod_{v \in S_U} q_v, \label{prasad_here} \\
  &[\Gamma:\Gamma_U] \le C_G^{|S_U|} N_{k/\QQ}(\Delta_{\ell/k})^{\frac 1 2} \Delta_k^{\frac 3 2} \label{index_here}
\end{align}
where $C_G$ is a constant depending only on $G$ (we take the largest between the two statements; in the sequel we will keep the notation even though we might have to increase the constant). 

\begin{lemma} \label{last_bound2}
  The number of $h \in G$ such that $h^m = g$ for a regular $g \in G$ is bounded by a constant depending only on $G$, the degree $[k:\QQ]$ and $m$. 
\end{lemma}

\begin{proof}
  For a regular element, its $m$--th roots belong to a torus and hence their number is bounded by a constant depending on $\dim(G)$ and the number of roots of unity in a finite extension of $k$ of degree at most the rank of $G$. 
\end{proof}

Note that the $m$-th power of a strongly $m$-regular element is strongly regular. It follows from the lemma above that the map $\gamma \mapsto \gamma^m$ from the set of semisimple, $m$-regular elements of $\Gamma$ to strongly regular elements in $\Gamma_U$ is at most $C_G$-to-one. On the other hand, it follows from the triangle inequality, that for any isometry $g \in G$ we have
\[
1_{B(R)}(x^{-1}gx) \le 1_{B(mR)}(x^{-1}g^mx).
\]
For any $g \in \Gamma \cap W$, the power $\gamma^m$ is regular, so $G_{\gamma^m} = G_\gamma$.  We conclude that 
\[
\vol(\Gamma_\gamma \bs G_\gamma) = \vol(\Gamma_{\gamma^m} \bs G_{\gamma^m}) \le \vol\left((\Gamma_U)_{\gamma^m} \bs G_{\gamma^m}\right).
\]
Putting this together we get
\[
\tr R_\Gamma^W 1_{B(R)} \le C_G \cdot \tr R_{\Gamma_U}^{W'} 1_{B(mR)}, 
\]
where $W'$ is the set of strongly regular, $\RR$-regular elements of $G$. By Proposition \ref{principal_case}, it follows that for any $\eps > 0$ we have
\[
\tr R_\Gamma^W 1_{B(R)} \ll_{G, R, \eps, d} \Delta_k^{\frac {\dim\G - (\dim\G - \dim\T)} 2 - \eps}.
\]
Using \ref{prasad_here} we get
\begin{align*}
  \tr R_\Gamma^W 1_{B(R)} &\ll_{G, R, \eps, d} \frac{\vol(\Gamma_U \bs G)}{\Delta_k^{\frac{\dim \G - \dim\T}2 - \eps} \cdot N_{k/\QQ}(\Delta_{\ell/k}) \cdot \prod_{v \in S_U} q_v} \\
  &= \frac{[\Gamma:\Gamma_U]}{\Delta_k^{\frac{\dim \G - \dim\T}2 - \eps} \cdot N_{k/\QQ}(\Delta_{\ell/k}) \cdot \prod_{v \in S_U} q_v} \cdot \vol(\Gamma \bs G).
\end{align*}
Now using \eqref{index_here} and the fact that $C^{|S|} \ll \prod_{v \in S} q_v^{1/2}$ for any $C>1$, we deduce that 
\[
\tr R_\Gamma^W 1_{B(R)} \ll_{G, R, \eps, d} \frac 1{\Delta_k^{\frac{\dim \G - \dim\T - 3}2 - \eps} \cdot N_{k/\QQ}(\Delta_{\ell/k})^{1/2} \cdot \prod_{v \in S_U} q_v^{1/2}} \cdot \vol(\Gamma \bs G).
\]
When $\G$ is not of type $A_1$ we have $\dim\G - \dim\T \ge 4$ and we see that \eqref{nbcc_for_maximal} holds, since for any $C$ the denominator $\Delta_k^{\frac{\dim \G - \dim\T - 3}2 - \eps} \cdot N_{k/\QQ}(\Delta_{\ell/k})^{1/2} \cdot \prod_{v \in S_U} q_v^{1/2}$ is $\le C$ only for finitely many choices of $k, \ell$ and $S_U$. The remaining case where $\G$ is of type $A_1$ is dealt with in \cite{Fraczyk}; we could also deduce it from our arguments here, going back to the proof of \eqref{index_here} and observing that we can get the better bound $[\Gamma:\Gamma_U] \le C_G^{|S_U|}\Delta_k^{\frac 1 2}$ for type $A_1$, since $\G$ is always inner so $\ell=k$.


\subsection{The case of non-uniform lattices} \label{nonuniform}

The proof of Theorem \ref{Main_bounded_degree} when all $\Gamma_n$ are non-uniform is exactly the same as in the case where they are uniform except that we need in addition to choose the conjugacy-invariant, sufficiently dense set $W$ such that for any choice of a lattice $\Gamma$ in $G$ and any $\gamma \in \Gamma \cap W$ the quotient $\Gamma_\gamma \bs G_\gamma$ is compact. The following lemma shows that for this we can use the set of $m$-regular elements. 
 
\begin{lemma}\label{strreg}
  Let $G$ be a semisimple Lie group. There exists $m=m(G)$ such that every non-uniform, irreducible arithmetic lattice $\Gamma$ and every strongly $m$-regular element $\gamma\in \Gamma$ the group $\Gamma_\gamma$ is a lattice in $G_\gamma$.
\end{lemma}

\begin{proof}
Let ${\bf G}$ be an algebraic group over a number field $k$ such that we have a surjective map ${\bf G}(k_\infty)\to G$ and $\Gamma \subset G(k)$. The lattice is non-uniform and irreducible so $\G$ is not $k$-anisotropic. It follows that $G\simeq {\bf G}( k_\infty)$, because the latter has no compact factors. Let $\gamma$ be a strongly regular element of $\G(k)$. The group $\Gamma_\gamma$ is a lattice in $G_\gamma$ if and only if ${\bf G}_\gamma$ is $k$-anisotropic.  Write ${\bf T} := {\bf G}_\gamma$, it is an algebraic torus, by Proposition \ref{p-ConnectedCentralizer}. The torus ${\bf T}$ is anisotropic if and only if it has no rational characters. The latter means that the action of ${\rm Gal}(\overline k/k)$ on $X^*({\bf T})$ has no non-trivial fixed points. 
Write $T={\bf T}(k_\infty)$ and let $X^*(T)$ be the module of complex characters of $T$. Then 
$$G\simeq \prod_{\nu}{\bf G}( k_\nu), \ \ T\simeq \prod_{\nu}{\bf T}(k_\nu)\textrm{ and } X^*(T)\simeq \bigoplus_{\nu}X^*({\bf T}),$$
where $\nu$ runs over the archimedean places of $k$.

\textbf{Claim 1.} There exist $m_0$ depending only on $G$ such that if $X^*({\T})^{{\rm Gal}(\overline k/k)}\neq 0$, then there is a nonzero $\Gal(k)$-fixed character $\xi$ which is a product of at most $m_0$ roots. 

To prove this claim we observe that the set of $\Gal(k)$-fixed characters contains the image of the Galois action of the element $\sum_{g \in \Gal(L/k)} g$ where $L$ is a splitting field of $\T$. Hence there is at least one root for which this is nonzero; since the Galois action preserves the root system, it follows that we can take $m_0 = [L:k]$. On the other hand an upper bound for the degree of a splitting field depends only on $\dim(\T)$ so this finishes the proof. 


\textbf{Claim 2.} Let $\xi$ be a rational character of ${\bf T}$ with $\|\xi\|\leq m_0$. Write $\xi_\nu$ for the extension to $\T(k_\nu)$. Then the character 
$$\lambda:=\prod_{\nu \textrm{ real}}\xi_\nu\prod_{\nu \textrm{ complex}}(\overline{\xi_\nu}\xi_\nu)$$ satisfies $\|\lambda\|\leq [k:\QQ]m_0$ and $\lambda(\gamma)=\pm 1$.

The first assertion holds because $\|\cdot\|$ is sub-additive. For the second, observe that $\lambda(\gamma)=N_{k/\QQ}(\xi(\gamma))$ and that $\xi(\gamma)\in \mathcal O_k^\times$ because $\gamma\in\Gamma$.

\medskip

We now conclude the argument for the proof of the lemma. Suppose $\gamma$ is such that ${\bf G}_\gamma$ is not anisotropic. By combining the claims we construct a nontrivial character $\lambda^2$ with $\|\lambda^2\|\leq 2[k:\QQ]m_0$ such that $\lambda^2(\gamma)=1$. Non-uniform lattices have trace field of degree at most $\dim G$, so the lemma holds with $m=2m_0\dim G$. 
\end{proof}


\section{Local and global estimates for orbital integrals} \label{s:orbint}

This section is devoted to the proof of the following theorem. We note that similar results are proven for instance in \cite[Proposition 3.13]{Kim_Shin_Templier} but they are not sufficient for our purposes (we need to have no multiplicative constant in our local inequalities). 

\begin{theorem} \label{global_orbint}
  Let $k$ be a number field, $\G$ a simply connected, semisimple $k$-group, $U$ a compact-open subgroup in $\G(\Ade_f)$ which is a product of parahoric subgroups. Fix the normalisation of Haar measure on $\G(\Ade_f)$ so that $\vol(U)=1$.  Let $f \in C_c^\infty(\G(k_\infty))$ and let $f_\Ade = f \otimes 1_U$. There exists $C_o$, depending on $R, [k:\QQ]$, such that for any strongly regular, $\RR$-regular element $\gamma \in \G(k) \cap U$ with $m(\gamma) \le R$ we have
  \[
  \mathcal O(\gamma, f_\Ade) \le C_o \|f\|_\infty.
  \]
  
\end{theorem}

\begin{proof}
  We prove Theorem \ref{global_orbint} assuming the local estimates given later. By assumption, we have $U = \prod_{v \in \val_f} U_v$ where $U_v$ is a parahoric subgroup of $\G(k_v)$. It follows that 
  \[
  \mathcal O(\gamma, f_\Ade) = \mathcal O(\gamma, f_\infty) \cdot \prod_{v \in \val_f} \mathcal O(\gamma, 1_{U_v}).
  \]
  We can apply Proposition \ref{est_orbint_nona} and Proposition \ref{est_orbint_a} together with Lemma \ref{partial_to_full_discriminant} to estimate the right-hand side and get
  \[
  \mathcal O(\gamma, f_\Ade) \le C_R \|f\|_\infty \cdot N_{k/\QQ}(\Delta(\gamma))^a,
  \]
  for some $a > 0$. We have that $m(\Delta(\gamma)) \ll m(\gamma) \le R$, and since $\Delta(\gamma)$ is an algebraic integer this gives an upper bound on $N_{k/\QQ}(\Delta(\gamma))$ depending only on $[k:\QQ]$ and $R$. This finishes the proof. 
\end{proof}


\subsection{Nonarchimedean orbital integrals}

For this subsection we let $v \in V_f$, so $k_v$ is a local non-archimedean field. We denote the cardinality of the residue field of $k_v$ by $q$. We fix a simply connected semisimple $k_v$-group $\G$.

Let $U_v$ be a maximal compact subgroup of $\G(k_v)$. We fix the Haar measure on $\G(k_v)$ so that $\vol(U_v) = 1$. Let $\gamma \in U_v$ be a semisimple regular element. Let $\G_\gamma$ be the centraliser of $\gamma$ in $\G$. It is an algebraic subgroup whose identity component is a torus. Our goal is to estimate the following integral orbital :
\[
\mathcal O(\gamma, 1_{U_v}) = \int_{\G_\gamma(k_v) \bs \G(k_v)} 1_{U_v}(x^{-1}\gamma x) dx. 
\]
Namely, we will prove the following result. 

\begin{proposition} \label{est_orbint_nona}
  There exists a constant $a > 0$, which depends only on the absolute type of $\G$, such that for any $\gamma \in \G(k_v)$ as above we have:
  \[
  \mathcal O(\gamma, 1_{U_v}) \le |\Delta(\gamma)|_v^{-a}.
  \]
\end{proposition}

\begin{proof}
  The main ingredient in the proof is the following relation between the orbital integrals and Bruhat-Tits buildings. Recall that $X = X(\G, k_v)$ is the Bruhat--Tits building associated with the simply-connected group $\wdt\G$ over $k_v$ (see \ref{BT_buildings}). We will prove that there exists $b$ depending only on absolute type of $\G$ such that 
  \begin{equation} \label{majo_fixedball}
    \mathcal O(\gamma, 1_{U_v}) \le \left| B_X(x_0, R) \right|, \, R = b\cdot v(\Delta(\gamma))
  \end{equation} 
  
  where $B_X(x_0, R)$ is the set of vertices in the ball of radius $R$ around $x_0$ in the building $X$. By the estimates in \cite[Section 3.3]{Leuzinger}, there exists constants $c>0$ depending only on the absolute type of $\G$ such that $|B_X(x_0, R)| \le q^{cR}$ for $R \ge 1$. Together with \eqref{majo_fixedball} this implies that
  \[
  \mathcal O(\gamma, 1_{U_v}) \le  q^{-cb \cdot v(\Delta(\gamma))} 
  \]
  and this is equal to $|\Delta(\gamma)|_v^{-a}$ where $a$ depends only on the absolute type of $\G$.
\end{proof}

The rest of this subsection is dedicated to the proof of \eqref{majo_fixedball}; we deal first with the case where $\G_\gamma$ is split over $k_v$ and then deduce the general case by using embedding arguments. 


\subsubsection{}

Recall that we assume that the Haar measures on $\G(k_v), \G_\gamma(k_v)$ are normalised so that the maximal compact subgroups $U_v, \G_\gamma(k_v)^b$ have measure 1. In the proofs below we will use the shorthand $\mathcal O(\gamma, 1_{U_v}) = \mathcal O_\gamma$. We start with giving a more geometric formula for $\mathcal O_\gamma$. We assume that a chamber $C$ of $X$ has been fixed. Recall that the types of facets in $X$ are indexed by the facets of $C$. For a local type $\Theta$, we denote by $X_\Theta$ the set of translates of the facet $C_\Theta$ under $\G(k_v)$. 

\begin{lemma} \label{estimate_balls_orbint}
  Let $U_v$ be of type $\Theta$ and let $(X_\Theta)^\gamma$ the simplices of type $\Theta$ in $X$ which are fixed by $\gamma$. Then 
  \[
  \mathcal O(\gamma, 1_{U_v}) = \sum_{x \in \G_\gamma(k_v) \bs (X_\Theta)^\gamma} |\G_\gamma(k_v)^b \cdot x|.
  \]
\end{lemma}

\begin{proof}
  Let $\alpha$ be a function on $\G(k_v)$ such that $\int_{\G_\gamma(k_v)} \alpha(tx) dt = 1$ for any $x \in \G(k_v)$. It follows that 
  \begin{align*}
    \int_{\G_\gamma(k_v) x U_v} \alpha(x) dx &= |\G_\gamma(k_v) \bs \G_\gamma(k_v) x U_v| \\
    &= [G_\gamma(k_v)^b:G_\gamma(k_v)^b \cap xU_vx^{-1}]. 
  \end{align*}
  Now $1_{U_v}(x^{-1}\gamma x)$ is nonzero if and only if $\gamma$ fixes the corresponding coset $xU_v$. So 
  \begin{align*}
    \int_{\G(k_v)} \alpha(x)1_{U_v}(x^{-1}\gamma x) dx &= \sum_{xU_v \in G(k_v)/U_v : \gamma x = x} \int_{xU_v} \alpha(y) dy \\
    &= \sum_{\substack{\G_\gamma(k_v)xU_v \in \G_\gamma(k_v)\bs G(k_v)/U_v \\ \gamma x = x}} [G_\gamma(k_v)^b:G_\gamma(k_v)^b \cap xU_vx^{-1}]
  \end{align*}
  which finishes the proof by identifying $G(k_v)/U_v$ with $X_{\Theta_v}$ and $[G_\gamma(k_v)^b:G_\gamma(k_v)^b \cap xU_vx^{-1}]$ with $|\G_\gamma(k_v)^b \cdot x|$. 
\end{proof}


\subsubsection{Proof of \eqref{majo_fixedball} in the split case}

We assume here that the maximal torus $\G_\gamma$ is split over $k_v$, so there is a unique apartment $A_\gamma$ of $X$ stabilised by $\G_\gamma(k_v)$. 

\begin{lemma} \label{split_case}
  Under the assumptions above there exists $c'$ depending only on the absolute type of $\G$ such that any simplex of $X$ fixed by $\gamma$ lies at most at distance $c' \cdot v(\Delta(\gamma))$ from $A_\gamma$. 
\end{lemma}

\begin{proof}
  Let $C_\Theta$ be the simplex of $X$ fixed by $U_v$. We choose a chamber $C$ of $A_\gamma$ containing $C_\Theta$, so there exists a special vertex $x_0 \in C$ fixed by $\gamma$. Let $I \le U_v$ be the Iwahori subgroup fixing $C$ pointwise. Let $\Phi$ be the (linear) root system of the pair $(\G, \G_\gamma)$. The chamber $C$, together with the vertex $x_0$, determine a set $\Phi^+$ of positive roots. Let $\N$ be the associated maximal unipotent subgroup of $\G$. Then according to \cite[3.3.2]{Tits_Corvallis} we have the Iwasawa decomposition
  \[
  \G(k_v) = \G_\gamma(k_v) \cdot \N(k_v) \cdot I.
  \]
  Let $C_\Theta' = gC_\Theta$ 
  and write $g = ank$ the Iwasawa decomposition of $g$. If $\gamma$ preserves $C_\Theta'$ then we have $\gamma an \in an U_v$, and as $a \in \G_\gamma(k_v)$ it follows that $n^{-1}\gamma n \in U_v$ and finally that $\gamma^{-1}n^{-1}\gamma \cdot n \in U_v$. 

  For $\lambda\in \Phi$ we denote by $\N_\lambda$ the 1-parameter unipotent subgroup of $\G$ and associated to $\lambda$. We will also use the notation set up in \cite{Tits_Corvallis}: if $\lambda \in \Phi$ and $k \in \ZZ$ we let $\lambda + k$ be the affine function on $A_\gamma$ with vector part $\lambda$ and such that $(\lambda+k)(x_0) = k$. For such a function $f = \lambda + k$ let $X_f$ be the subgroup of $\N_\lambda(k_v)$ fixing the half-apartment $f^{-1}(\interval f0{+\infty}o)$. This determines an isomorphism $n_\lambda$ from the additive group $k_v$ to $\N_\lambda$ such that $n_\lambda(\mo_{k_v}) = X_\lambda$.

  We can describe the chamber $C_\Theta$ as follows :
  \[
  C_\Theta = \bigcap_{\lambda \in \Phi^+ \setminus \Theta} \left( \lambda^{-1}(\interval f0{+\infty}o) \cap (1-\lambda)^{-1}(\interval f0{+\infty}o) \right) \cap \bigcap_{\lambda\in\Theta} \lambda^{-1}(\{0\}) 
  \]
  and it follows from \cite[3.1.1]{Tits_Corvallis} that the product map defines a bijection
  \[
  U_v = \left( \prod_{\lambda \in \Phi^+} X_\lambda \right) \times \G_\gamma(k_v)^b \times \left( \prod_{\lambda \in \Phi^+ \setminus\Theta} X_{1 - \lambda} \times \prod_{\lambda\in\Theta} X_{-\lambda} \right). 
  \]
  As $\gamma^{-1}n\gamma \cdot n \in \prod_{\lambda\in\Phi^+} N_\lambda(k_v)$, we get that $\gamma^{-1}n\gamma \cdot n \in U_v$ is equivalent to it belonging to $\prod_{\lambda \in \Phi^+} X_\lambda$. We choose a numbering $\Phi^+ = \{\lambda_1, \ldots, \lambda_k\}$ satisfying the condition in Lemma \ref{comm_ss_unip} and we write
  \[
  n = n_{\lambda_k}(t_k) \cdots n_{\lambda_1}(t_1).
  \]
  Let $b_{j, s}^i$ be the structure constants of a Chevalley basis of the Lie algebra of $\G$, determined by $[u_j, u_s] = b_{j, s}^i u_i$ if $u_i$ are the basis vectors corresponding to the roots $\lambda_i$. They are integers depending only on the absolute type of $\G$, see \cite[Theorem 1 on p.~6]{Steinberg_notes}. It follows from Lemma \ref{comm_ss_unip} that (for concision we use the notation $a_i = \lambda_i(\gamma)$):
  \begin{equation} \label{coefficients_unip}
    n^{-1} \cdot \gamma n \gamma^{-1} = \prod_{i=1}^k n_{\lambda_i} \left( (a_i-1)t_i + \sum_{\lambda_{j_1} + \cdots + \lambda_{j_m} = \lambda_i} b_{j_1, \ldots, j_m}^i \prod_{l=1}^m t_{j_l} \right),
  \end{equation}
  where $b_{j_1, \ldots, j_k}^i$ is a nonzero product of the $b_{i, j}^k$ (depending on the chosen order on roots). 
  We have arrived at the following reformulation : 
  \[
  n^{-1} \gamma^{-1}n\gamma \in I \: \Longleftrightarrow \: \forall i = 1, \ldots, k : \left( (a_i-1)t_i + \sum_{\lambda_{j_1} + \cdots + \lambda_{j_m} = \lambda_i} b_{j_1, \ldots, j_m}^i \prod_{l=1}^m t_{j_l} \right) \in \mo_{k_v}. 
  \]
  We want to deduce from this the lower bounds for $v(t_i)$, $1 \le i \le k$. Note that if $\N$ were abelian, this would be immediate, as all $b_{j,s}^i$ would vanish and $(a_i-1)t_i \in \mo_{k_v}$ gives $v(t_i) \ge -v(a_i-1)$.

  \medskip

  {\noindent\bf Type $A_2$} We give first a simple example to illustrate the general principle, and point to a possible sharpening of our estimates. If $\G = \SL_3$ we can choose the $\lambda_i, n_i$ so that all $b_{j,s}^i = 0$ for $i = 1, 2$ and $b_{1,2}^3 = 1$. We get that
  \[
  v(t_i) \ge -v(a_i-1),\, i = 1, 2
  \]
  and
  \[
  v(t_3) \ge -v(a_3-1) \text{ or } v((a_3-1)t_3) = v(a_2t_1t_2) \ge v(t_1)
  \]
  (for the last inequality we use $v(a_2) \ge v(a_2-1) \ge -v(t_2)$) so in the worst case we get that
  \[
  v(t_1) \ge -v(a_1-1), \, v(t_2) \ge -v(a_2-1), \, v(t_3) \ge -v(a_3-1) - v(a_1-1)
  \]
  which implies that $v(t_i) \ge -v(\Delta(\gamma))$ for all $i$, and also $\sum_i v(t_i) \ge -2v(\Delta(\gamma))$. 

  \medskip

  {\noindent\bf General case} Ordering the $\lambda_i$s (as in Lemma \ref{comm_ss_unip}) so that the $l$--th term in the lower central series of $\N$ is spanned by $\N_{\lambda_{k_l}}, \ldots, \N_{\lambda_{k_{l+1}-1}}$ (with $1 = k_1 \le k_2 \le \cdots \le k_m$), we see that all $b_{j, s}^i = 0$ for $1 \le i < k_1$, so that $v(t_i(a_i-1)) \ge 0$, hence $v(t_i) \ge -v(a_i-1)$ for those $i$s. 

  Now we proceed by induction on $l$ to estimate $v(t_i)$ for $k_l \le i \le k_{l+1} - 1$. Assume that we know that for $1 \le i \le k_l-1$ we have $v(t_i) \ge -\sum_{1 \le j \le k_l-1} b_j v(a_j - 1)$ for some nonnegative integers $b_j \le 2^{l-1}$. Now \eqref{coefficients_unip} implies that either $v((a_i-1)t_i) \ge 0$, or
  \[
  v((a_i-1)t_i) \ge - v\left(\sum_{\lambda_{j_1} + \cdots + \lambda_{j_m} = \lambda_i} b_{j_1, \ldots, j_m}^i \prod_{s=1}^m a_{j_s}^{\eps_{j_1, \ldots,j_k}^i} t_{j_s} \right). 
  \]
  In the first case we get that $v(t_i) \ge -v(a_i-1)$ and we are done. In the second case we get that
  \begin{align*}
    v(t_i) &\ge -v(a_i-1) + \min_{\lambda_{j_1} + \cdots + \lambda_{j_m} = \lambda_i} v\left( \prod_{s=1}^m a_{j_s}^{\eps_{j_1, \ldots,j_k}^i} t_{j_s} \right) \\
    &\ge -v(a_i-1) + \min_{\lambda_{j_1} + \cdots + \lambda_{j_m} = \lambda_i} \sum_{s=1}^m v(t_{j_s})
  \end{align*}
  and by the recursion hypothesis it follows that :
  \[
  v(t_i) \ge -v(a_i-1)  - 2\sum_{1 \le j \le k_l-1} b_jv(a_j - 1) \ge \sum_{1 \le j \le k_{l+1}-1} b_j'v(a_j - 1)
  \]
  with $b_j' \le 2^l$. 

  \medskip

  In conclusion: from the condition that $n^{-1} \cdot \gamma n \gamma^{-1} \in I$, equivalent to $(a_i-1)t_i + \sum_{\lambda_j + \lambda_l = \lambda_i} a_{j,l}^it_it_j \in \mo_{k_v}$ for all $1 \le i \le k$ we have arrived at the following set of inequalities :
  \[
  v(t_i) \ge -2^{m-1} \sum_j v(a_j - 1)
  \]
  Since
  \[
  -v(\Delta(\gamma)) = -v\left(\prod_{\lambda\in \Phi}(1 - \lambda(\gamma)) \right) \le -\sum_{\lambda\in \Phi^+} v(1 - \lambda(\gamma)),
  \]
  (the inequality following from $-v(1+\lambda(\gamma)) \ge -v(\lambda(\gamma)) \ge 0$ for $\lambda \in \Phi^+$) we get that
  \begin{equation}
    v(t_i) \ge -c \cdot v(\Delta(\gamma)), \: 1 \le i \le k 
  \end{equation} 
  where $c = 2^m$ depends only on the absolute type of $\G$. 

  \medskip

  We can finally conclude. Let $x_1$ be the vertex of $A_\gamma$ at the tip of the intersection of the half-apartments $\lambda_i^{-1}\{-v(\Delta(\gamma))\}$. We have just proven that if $\gamma$ fixes $gI$, with $g = ank$, then $n$ fixes $x_1$. In particular we get that 
  \[
  d(A_\gamma, gC) = d(A_\gamma, nC) \le d(x_1, nC) = d(x_1, C) \le d(x_1, x_0) = c' \cdot v(\Delta(\gamma)). 
  \]
  where $c'$ depends only on $c$ and the geometry of the apartment, so only on the absolute type of $\G$. As a fixed vertex belongs to a fixed chamber at the same distance from $A_\gamma$, this finishes the proof of the lemma.  
\end{proof}


Now let
\[
R := \sup_{x \in \Fix_{X^\Theta}(\gamma)} d(x, A_\gamma) \le c \cdot v(\Delta(\gamma)),
\]
the inequality following from by Lemma \ref{split_case}. 

The group $\G_\gamma(k_v)$ acts transitively on $A_\gamma^\Theta := A_\gamma \cap X^\Theta$ (this follows from Iwasawa decomposition). Moreover every vertex of $A_\gamma$ is at distance at most 1 from a vertex of $A_\gamma^\Theta$ (because every chamber contains a vertex of every type)It follows (as $2R \ge R+1$ for $R \ge 1$) that there exists a set $\{y_1, \ldots, y_n\}$ of representatives for the orbits of $\G_\gamma(k_v)$ on $\Fix_{X^\Theta}(\gamma)$ such that $y_i \in B_{X}(x_0, 2R)$ for all $i$. The subgroup $\G_\gamma(k_v)^b$ preserves this ball and the orbits $\G_\gamma(k_v)^b \cdot y_i$ are pairwise disjoint, so it follows that : 
\[
\sum_{x \in \G_\gamma(k_v) \bs \Fix_{X^\Theta}(\gamma)} \left| \G_\gamma(k_v)^bx\right| \le B_{X}(x_0, 2R).  
\]
Using Lemma \ref{estimate_balls_orbint}, this finishes the proof of \eqref{majo_fixedball} in the case where $\G$ is $k_v$-split. 


\subsubsection{Proof of \eqref{majo_fixedball} in the remaining cases}

Let $L_v/k_v$ be a finite Galois extension splitting $\G_\gamma$. Then there is an inclusion of buldings $X \subset X(\wdt G, L_v) = X_{L_v}$ such that the image of $X$ is a totally geodesic subspace. The Galois group $\Gal(L_v/k_v)$ acts by isometries on $X_{L_v}$. 

\medskip

\noindent{\bf Tame case.} We suppose here that $L_v/k_v$ is unramified or tamely ramified. In these cases we have $X = X_{L_v}^{\Gal(L_v/k_v)}$ by \cite[2.6.1]{Tits_Corvallis}, \cite[Proposition 5.1.1]{Rousseau_thesis} (see also \cite{Prasad_galois}). The Galois group preserves $A_{\gamma, L_v}$ where it acts by affine isometries, so $A_{\gamma, L_v} \cap X = (A_{\gamma, L_v})^{\Gal(L_v/k_v)}$ is an affine subspace. The linear part of the Galois action on $A_{\gamma, L_v}$ is given by the action on the co-root space $X_*(\G_\gamma)$. It follows that the dimension of its fixed points equals the $k_v$-rank of $\G_\gamma$, which equals the $k_v$-rank of $\G_\gamma$. 

We define $A_\gamma := A_{\gamma, L_v} \cap X$, which by the discussion above is an affine subset of dimension $r$ in an apartment of $X$. It follows from Lemma \ref{split_case} that $\Fix_{X}(\gamma)$ is contained in the $c'\cdot v(\Delta(\gamma))$-neighbourhood of $A_\gamma$, with respect to the metric induced on $X$ by that of $X_{L_v}$. This metric might not be the intrisic metric on $X$, but it can only be multiplied by a factor depending only the Galois action on $X_*(\G_\gamma)$, which is determined up to finite ambiguity by the absolute type of $\G$. In particular the statement of the lemma remains valid in this case.

We let $R = \sup(d(x, A_\gamma) : x \in \Fix_{X_\Theta}(\gamma))$, by the preceding comments we have that $R \le c' \cdot v(\Delta(\gamma))$. In case $R = 0$ (that is $v(\Delta(\gamma)) = 0$) the set $X_\Theta \cap A_\gamma$ is nonempty, and the action of $\G_\gamma(k_v)$ on it is transitive. It follows that $\mathcal O_\gamma = 1$. 

In case $R \ge 1$ the set $X_\Theta \cap A_\gamma$ might be empty. We fix representatives $x_1, \ldots, x_s$ of the $\G_\gamma(k_v)$-orbits on $A_\gamma$, we may assume that $x_i \in B_{X}(x_1, 1)$ for $1 \le i \le s$. Since $\Fix_{X_\Theta}(\gamma)$ is contained in the $R$-neighbourhhod of $A_\gamma$, using Lemma \ref{estimate_balls_orbint} we get that
\[
\mathcal O_\gamma \le |B_{X}(x_1, R+1)| \le |B_{X}(x_1, 2R)|.
\]
This finishes the proof in case $\G_\gamma$ is split by a unramified or tamely ramified extension of $k_v$. 


\medskip

\noindent{\bf Wild ramification.} Assume now that $L_v/k_v$ is wildly ramified. This case is similar to the preceding but the set $A_{\gamma, L_v} \cap X$ itself might be empty. We note that we must have $v(\Delta(\gamma)) \ge 1$. We fix a vertex $x_0 \in \Fix_{X_k}(\gamma)$ which is as close as possible (in $X_{L_v}$) to $A_{\gamma, L_v}$. Then $\Fix_{X_\Theta}(\gamma)$ is contained in the $(R-w)$-neighbourhood of $G_\gamma(k_v) \cdot x_0$ where $w = d(x_0, A_\gamma)$. We can then apply the exact same argument as above to prove \eqref{majo_fixedball}. 


\subsection{Archimedean orbital integrals}

In this section $k_v = \RR$ or $k_v = \CC$, as above $\G$ is a semisimple $k_v$-group. $f$ is a function with support in the ball $B_{\G(k_v)}(1, R)$ of radius $R > 0$ around the identity, in a fixed Euclidean norm on $\G(k_v)$). We want to estimate the orbital integrals 
\[
\mathcal O(\gamma, f) = \int_{\G_\gamma(k_v) \bs \G(k_v)} f(x^{-1}\gamma x) dx
\]
for sufficiently regular $\gamma$. For a regular element $\gamma \in \G(k_v)$ and $\Phi_\RR$ the root system of $\G(k_v)$ relative to $\G_\gamma(k_v)$ we denote
\[
\Delta^+(\gamma) = \prod_{\lambda\in \Phi_\RR} (1 - \lambda(\gamma)). 
\]

\begin{proposition} \label{est_orbint_a}
  Let $\gamma \in \G(k_v)$ be a $\RR$-regular element $\gamma \in \G(k_v)$. There exists $C > 0$, depending only on $\G(k_v)$, on $d_G(1, \gamma)$ and on $R$, such that: 
  \[
  \mathcal O(\gamma, f) \le C \cdot \| f \|_\infty \cdot |\Delta^+(\gamma)|.
  \]
\end{proposition}

\begin{proof}
  Let $G = \G(k_v)$, $T = \G_\gamma(k_v)$ 
  and let $P$ be the minimal parabolic subgroup of $G$ containing the connected component of $T$. Let $A$ be a maximal $\RR$-split sub-torus of $T$. Then $P$ has a Langlands decomposition $P = MAN$, with $N$ unipotent and $M$ semisimple. Since we assume $\gamma$ to be $\RR$-regular, $A$ is a maximal $\RR$-split torus of $\G$ by Lemma \ref{lem:Rregular}  and $M$ is in fact compact.

  Let $K$ be a maximal compact subgroup of $G$ containing $M$, such that $A$ is fixed by the Cartan involution of $G$ associated with $K$. Then we have the Iwasawa decomposition $G = KAN$ (see \cite[Proposition 7.31]{Knapp_book_lie}). 

  We have
  \begin{align*}
    \mathcal O(\gamma, f) &= \int_{A \bs G / K} \int_K f(u^{-1}x^{-1}\gamma xu) du dx \\
    &= \int_{A \bs G / K} \tilde f(x^{-1}\gamma x) dx
  \end{align*}
  where $\tilde f(g) = \int_K f(u^{-1}gu) du$. As $\tilde f$ satisfies $\|\tilde f\|_\infty \le \|f\|_\infty$ and its support is contained in $B(1, R + d)$ where $d$ is the diameter of $K$, we will assume in the rest of the proof that $f = \tilde f$.

  We obtain
  \[
  \mathcal O(\gamma, f) = \int_{N} f(n^{-1}\gamma n) dn = \int_{N} f(\gamma \cdot y(n)) dn
  \]
  where $y(n) = \gamma^{-1}n^{-1}\gamma \cdot n$. Let $\Phi_\RR^+$ be the positive roots associated with $N$ in the relative root system of $(G, A)$. By the proof of Lemma \ref{constantjacobian} we see that $y$ is a diffeomorphism of $N$ and its Jacobian equals $\prod_{\lambda \in \Phi_\RR^+} (1 - \lambda(\gamma)) = \Delta^+(\gamma)$. The Haar measure on $N$ is given by $\prod_{\lambda \in \Phi^+} dn_\lambda$ and it follows that
  \[
  \mathcal O(\gamma, f) = \Delta^+(\gamma) \int_{N} f(\gamma y) dy 
  \]
  We can conclude since
  \[
  \int_{N} \frac{f(\gamma y)}{\|f_\infty\|} dy \le \vol B_{N}(1, R+d_G(1, \gamma))
  \]
  and the right-hand side depends only on $\G(k_v)$, $R$ and $d_G(1, \gamma)$. 
\end{proof}


\subsubsection{All infinite places}

\begin{lemma} \label{partial_to_full_discriminant}
  Let $\G/k$ be a $k$-semisimple group. There exists a constant $C$ depending on the absolute type of $\G$, on $R$ and on $[k : \QQ]$ such that for any strongly regular, $\RR$-regular element $\gamma \in \G(k)$ with $m(\gamma) \le R$ we have: 
  \[
  \prod_{v \in \val_\infty} |\Delta^+(\gamma)|_v \le C \prod_{v \in \val_\infty} |\Delta(\gamma)|_v.
  \]
\end{lemma}

\begin{proof}
  We fix an absolute root system $\Phi$ for $(\G, \G_\gamma)$. For each archimedean place $v \in \val_\infty$ we get a relative root system $\Phi_\RR$, which is made up of the restrictions of elements of $\Phi$ to the split part of $\G_\gamma(k_v)$. It follows from \cite[Propositions 2.2-4]{Araki} that we can subdivide $\Phi$ into three subsets $\Phi_1, \Phi_2, \Phi_3$ such that :
  \begin{itemize}
  \item $\Phi_1 = \Phi_\RR \cap \Phi$ ;

  \item if $\lambda \in \Phi_2$ then $\ovl\lambda \in \Phi$ and $(\lambda \cdot \ovl\lambda)^{\frac 1 2} \in \Phi_\RR$ ; we choose a set $\Phi_2' \subset \Phi_2$ such that $\Phi_2$ is the disjoint union of $\Phi_2'$ and its complex conjugate. 

  \item if $\lambda \in \Phi_3$ then $\lambda$ takes imaginary values on the (real) Lie algebra of the $\RR$-split part of $\G_\gamma$. 
  \end{itemize}
  We have
  \begin{equation} \label{ydgyfqtdazgh}
  \frac{|\Delta^+(\gamma)|}{|\Delta(\gamma)|_v} = \prod_{\lambda \in \Phi_2'} \frac{1 -| \lambda(\gamma)|_v} {|(1 - \lambda(\gamma))(1 - \ovl\lambda(\gamma))|_v}  \cdot \prod_{\lambda \in \Phi_3} |1- \lambda(\gamma)|_v^{-1} \ll \prod_{\lambda \in \Phi_2 \cup \Phi_3} |1- \lambda(\gamma)|_v^{-1}
  \end{equation}
  where the estimate follows from the facts that 
  \[
  \forall z\in \CC \: \frac{1 -| z|} {|(1 - z)(1 - \ovl z)|} \ll \frac 1{|1-z|} + \frac 1{|1+z|}
  \]
  and if $\lambda \in \Phi_2$ then $-\lambda \in \Phi_2$ as well.
  
  Now we observe that for each $\lambda \in \Phi$, $1 - \lambda(\gamma)$ is a nonzero algebraic integer so that $\prod_{v \in \val_\infty} |1-\lambda(\gamma)|_v \ge 1$. On the other hand, since we are assuming that $m(\gamma) \ll R$ there exists some $M$ depending only on $R$ and $X$ such that $|1 - \lambda(\gamma)|_v \le M$ for all $v$. It follows that $|1- \lambda(\gamma)|_v \ge M^{-[k:\QQ]}$, so the product in the rightmost term of \eqref{ydgyfqtdazgh} is bounded, which establishes the lemma. 
\end{proof}



\section{Estimate for the number of elements of small height} \label{s:small}

We fix a semisimple Lie group $G$. The main result of this section is the following estimate. 



\begin{theorem}\label{t-SmallHeigtCount}
  For any $d \in \mathbb N$, $R > 0$ there exists a constant $C$ depending on $R, d$ and $G$ such that the following holds. For any number field $k$ with $[k:\QQ] \le d$, semisimple $k$-group $\G$ with $\G(k_\infty)$ isogenous to $G$ times a compact group, and any maximal compact subgroup $U=\prod_{v\in \val_f}U_v$ of $\G(\Ade_f)$, the number of conjugacy classes of strongly regular elements $\gamma \in \G(k)$ such that $m(\gamma) \le R$  and whose projection in $\G(\Ade_f)$ is contained in $U$ is bounded by $C$. 
\end{theorem}

We will use the following nomenclature : 
\begin{itemize}
\item A conjugacy class in $\G(k)$ will be called {\em rational} ;

\item A conjugacy class in $\G(\ovl k)$ will be called {\em geometric}. 
\end{itemize}

\subsection{Geometric conjugacy classes}

\begin{lemma} \label{charpoly_count}
  Let $R>0$ and $d \in \NN$. There exists $C$ such that for any totally real number field $k$ of degree $d$ and semisimple $k$-group $\G$ with $\G(k_\infty)$ isogenous to $G$ times a compact group, there are at most $C$ geometric conjugacy classes of semisimple elements in $\G(k) \cap U$ which contain a representative of Mahler measure at most $R$. 
\end{lemma}

\begin{proof}
  Let $\T$ be a maximal torus in $\G$. Every semisimple element of $\G(\ovl k)$ is conjugated into $\T(\ovl k)$, since all maximal tori are conjugated over $\ovl k$ and every semisimple element is contained in some maximal torus.

  Let $\Phi = \Phi(\G, \T)$ be the root system. For any $\gamma \in \G(k) \cap U$ the polynomial
  \[
  F_\gamma(t) = \prod_{\lambda \in \Phi} (t - \lambda(\gamma))
  \]
  characterises the $\G(\ovl k)$-conjugacy class of $\gamma$ up to finite ambiguity. It belongs to $\mo_k[t]$ and its Mahler measure is equal to that of $\gamma$. As $d = [k:\QQ]$ 
  we thus have that $\prod_{\sigma : k \to \CC} F_\gamma^\sigma$ is a monic polynomial in $\ZZ[t]$ with Mahler measure at most $dR$ and degree equal to $d(\dim(\G) - \rk(\G))$. By the Northcott property \cite[Theorem 1.6.8]{Northcott}, there are at most finitely many such polynomials and this finishes the proof of the lemma. 
\end{proof}


\subsection{Rational conjugacy classes in a geometric conjugacy class}\label{s-RatinGeo}

Let $\gamma$ be an element of $\G(k)$ and let $U=\prod_{v\in \val_f}U_v$ be a maximal compact subgroup of $\G(\Ade_f).$ Write 
\begin{align*}
  \mathcal S(\gamma) &= \{[\gamma']_{\mathbf G(k)} :\: \gamma'\in [\gamma]_{\mathbf G(\overline k)}\cap \mathbf G(k)\} ; \\
  \mathcal S_U(\gamma) &= \{[\gamma']_{\mathbf G(k)} :\: \gamma'\in [\gamma]_{\mathbf G(\overline k)}\cap \mathbf G(k),\, \gamma' \in U\}. 
\end{align*}

We have the following realisation of conjugacy classes as classes in Galois cohomology. 

\begin{lemma}\label{l-GCohomoParam1}
  There is a bijection $\mathcal S(\gamma) \to \ker[H^1(k,\mathbf G_\gamma)\to H^1(k,\mathbf G)]$. Explicitly, if $g\in \mathbf G(\overline k)$ is such that $[g^{-1}\gamma g]_{\mathbf G(k)}\in \mathcal S(\gamma)$ then we send it to the cohomology class $[\sigma\mapsto g^\sigma g^{-1}]\in H^1(k,\mathbf G_\gamma)$. 
\end{lemma}

\begin{proof}
  The set of $k$-points in the conjugacy class $[\gamma]_{\G(\overline k)}$ is naturally identified with $\G_\gamma\bs \G(k)$. The group $\G(k)$ acts on $\G_\gamma\bs \G(k)$ on the right and the set $\mathcal S(\gamma)$ is identified with the orbits of $\G(k)$. By exactness of the sequence \eqref{e-longcohomo} for $\mathbf H=\G_\gamma$, we get a bijection
  \[
  \mathcal S(\gamma)\simeq \ker[H^1(k,\G_\gamma)\to H^1(k,\G)].
  \]
  Explicitly, if $g\in \G(\overline k)$ is such that $[g^{-1}\gamma g]_{\G(k)}\in \mathcal C(\gamma,k)$ then it is mapped to the cohomology class $[\sigma\mapsto g^\sigma g^{-1}]\in H^1(k,\G_\gamma)$.
\end{proof}

In order to pinpoint the image of the subset $\mathcal S_U(\gamma)$ via the bijection from Lemma \ref{l-GCohomoParam1} we will need to use a local-global principle for the Galois cohomology.


\subsubsection{Local-global principle}

Let us define the local counterparts of the sets $\mathcal S(\gamma)$ and $\mathcal S_U(\gamma)$. If $v \in \val_\infty \cup \val_f$ let
\[
\mathcal S_v(\gamma) = \{[\gamma']_{\G(k_v)} :\: \gamma'\in [\gamma]_{\G(\overline k_v)}\cap \G(k_v)\}. 
\]
In addition, we define
\[
\mathcal S_{U_v}(\gamma) = \{[\gamma']_{\G(k_v)} :\: \gamma'\in [\gamma]_{\G(\overline k_v)}\cap \G(k_v),\, \gamma' \in U_v\}.  
\]
We have a commutative diagram 

\begin{equation}
\begin{tikzcd}
& & \Sha^1(\G_\gamma)\arrow[hookrightarrow]{d}\\ 
\mathcal S_U(\gamma)\arrow[hookrightarrow]{r}\arrow{d}{\phi} & \mathcal S(\gamma) \arrow[hookrightarrow]{r} \arrow{d}{\phi} & H^1(k,\G_\gamma) \arrow{d}{\phi} \\
\prod_{v}\mathcal S_{U_v}(\gamma)\arrow[hookrightarrow]{r} & \prod_{v}\mathcal S_v(\gamma) \arrow[hookrightarrow]{r}& \prod_{v}H^1(k_v,\G_\gamma)\\
\end{tikzcd}
\end{equation}


\begin{lemma}\label{l-GoodConjIneq1}
  Let $\gamma$ be an $m$-regular element. Then $|\mathcal S_U(\gamma)|\leq |\Sha^1(\G_\gamma)|\prod_{\nu}|\mathcal S_{U_\nu}(\gamma)|$.
\end{lemma}

\begin{proof}
By Proposition \ref{p-ConnectedCentralizer} the group $\G_\gamma$ is an algebraic torus so the map $\phi\colon H^1(k,\G_\gamma)\to \prod_{\nu} H^1(k_\nu,\G_\gamma)$ is a group homomorphism with kernel $\Sha^1(\G_\gamma)$. In particular it is $|\Sha^1(\G_\gamma)|$-to-one so 
\begin{equation}|\mathcal S_U(\gamma)|\leq \left|\phi^{-1}\left(\prod_{\nu}\mathcal S_{U_\nu}(\gamma)\right)\right|=|\Sha^1(\G_\gamma)|\prod_{\nu}|\mathcal S_{U_\nu}(\gamma)|.\end{equation}
\end{proof}


\subsubsection{Conjugacy classes at the archimedean places.}

Let $v \in \val_\infty$ and let $\gamma \in \G(k_v)$ be an element whose centraliser is a torus. By Lemma \ref{l-GCohomoParam1} and the Nakayama--Tate theorem \ref{t-LocalNakayamaTate} we have that $|\mathcal S_v(\gamma)| \le |H^1(\CC/k_v, X^*(\G_\gamma))|$. But since $\dim(\G_\gamma)$ is bounded independently of $k, v, \gamma$ we have that $|H^1(\CC/k_v, X^*(\G_\gamma))| \le 2^{\dim(\G_\gamma)} = O(1)$. In conclusion, since there are at most $d$ Archimedean places we have. 
\begin{equation} \label{product_real_places}
   \prod_{v \in \val_\infty} |\mathcal S_{k_v}(\gamma)| = O(1).
\end{equation}
We remark that this statement is in fact independent of our hypothesis that $[k:\QQ]$ remains bounded, since it can be proven that for Archimedean $v$ we have $|\mathcal S_{k_v}(\gamma)| = 1$ as long as $\G(k_v)$ is compact. 


\subsubsection{Conjugacy classes at the non-archimedean places.}

\begin{proposition}\label{p-GConjugacyLocalUnr}
  Let $v \in \val_f$ and let $\gamma$ be a regular semi-simple element of $\G(k_v)$ such that $|\Delta(\gamma)|_\frak p=1$ and $\G_\gamma$ is split by an unramified extension of $k_\frak p$. Then for any open compact subgroup $U_v$ of $\G(k_v)$ we have $|\mathcal S_{U_v}(\gamma)|=1$.
\end{proposition} 

\begin{proof}
  The first step is the following lemma. 

  \begin{lemma}\label{p-DiscCompactCocycles}
    Let $\gamma \in \G(k_v)$ such that $|\Delta(\gamma)|_v = 1$ and $\G_\gamma$ is split by a tamely ramified extension of $k_v$. Let $U_v$ be a maximal compact subgroup of $\G(k_v)$. The image of $\mathcal S_{U_v}(\gamma)$ in  $\ker[ H^1(k_v, \G_\gamma)\to H^1(k_v, \G)]$ consists only of compact classes. 
  \end{lemma}

  \begin{proof}
    Write $\T=\G_\gamma$. Let $L_v$ be a tamely ramified extension of $k_v$ splitting $\T$. Let $[\gamma']_{\G(k_v)}$ be a class in $\mathcal S_{U_v}(\gamma)$ with $\gamma'\in U_v$. Since $L_v$ splits $\T$, we have $H^1(L_v, \T) = \{1\}$ by Hilbert's theorem 90 and so $\gamma$ and $\gamma'$ are conjugate in $\G(L_v)$ according to Lemma \ref{l-GCohomoParam1}. 

    Choose $g\in \G(L_v)$ such that $\gamma'=g^{-1}\gamma g$. We need to show that the the image of $[\gamma']_{\G(k_v)}$ is a compact cohomology class. That amounts to showing that there exists $h \in \G(k_v)$ such that $h \in \T(L_v)g$ (so the cocycle associated with $h$ is cohomologous to that of $g$) and for all $\sigma \in \Gal(L_v/k_v)$ we have $h^\sigma h^{-1} \in \T(L_v)^b$ (so the cocycle is compact). 

    Let $X_{L_v}, X_{k_v}$ be the Bruhat--Tits buildings of $\G(L_v), \G(k_v)$ respectively. Then, as the extension $L_v/k_v$ is tamely ramified, it follows from \cite{Prasad_galois} that we have an inclusion of simplicial complexes $X_{k_v} \hookrightarrow X_{L_v}$ and $X_{k_v} = X_{L_v}^{\Gal(L_v/k_v)}$. 

    Let $\tau$ be the cell of $X_{k_v}$ stabilised by $U_v$ and $\widetilde U_v$ the stabiliser of $\tau$ in $\G(L_v)$; this is a compact subgroup of $\G(L_v)$. Since $\tau \in X_{L_v}^{\Gal(L_v/k_v)}$, the group $\widetilde U$ is stable under $\Gal(L_v/k_v)$. We have $ \tau = \gamma' \tau$, so $\gamma g\tau = g\tau$. By Lemma \ref{l-FixedSet1} below,  it follows that $g \in \T(L_v)\widetilde U_v$. We write $g = th$ with $h \in \widetilde U_v$, $t \in \T(L_v)$. Then, $h \in \T(L_v)g$ and $h^\sigma h^{-1} \in \widetilde U_v$ so that $h^\sigma h^{-1} \in \T(L_v)^b$. This concludes the proof. 
  \end{proof}

  \begin{lemma}\label{l-FixedSet1}
    If $\gamma \in \G(L_v)$ is regular, then there is a unique apartment $\mathcal{A} \subset X_{L_v}$ stabilized by $\G_\gamma$. If in addition $\gamma$ belongs to a compact subgroup and $|\Delta(\gamma)|_v = 1$, then we have that $\mathcal A = (X_{L_v})^\gamma$.

    Moreover if $g \in \G(L_v)$ and $\tau$ is a cell in $\mathcal{A}$ such that $h\tau \in \mathcal A$ as well then $h\in \G_\gamma(L_v)\Stab_{\G(L_v)}\tau$. 
  \end{lemma}
  
  \begin{proof}
    This follows immediately from  Lemma \ref{split_case}.
  \end{proof}

  Now we can finish the proof of the proposition. By Lemma \ref{p-DiscCompactCocycles}, the image of $\mathcal{S}_{U_v}(\gamma)$ consists only of compact cocycles in $H^1(k_v,\G_\gamma)$. The centralizer $\G_\gamma$ is split by an unramified extension of $k_v$, so by Proposition \ref{p-CompactClassTriv}, the only compact cohomolgy class in $H^1(k_v,\G_\gamma)$ is the trivial class. Hence $|\mathcal{S}_{U_v}(\gamma)|=1$. The proposition is proved.
\end{proof}

We will also need a bound to deal with the finite number of places where the hypotheses of Proposition \ref{p-GConjugacyLocalUnr} are not satisfied. These bounds are deduced immediately from Lemmas \ref{l-GCohomoParam1} and \ref{bound_H1_torus}. We record them in the following lemma. 

\begin{lemma}\label{c-ConjBoundNA}
  Let $v \in \val_f$ and let $\gamma \in U_v$ be a regular semisimple element. We have $|\mathcal{S}_{U_v}(\gamma)|\le C$ where $C$ depends only on the absolute type of $\G$. 
\end{lemma}


\subsection{Global conjugacy classes}\label{gcc}

We conclude here the proof of Theorem \ref{t-SmallHeigtCount}. Recall that $\G$ is a $k$-group with $\G(k_\infty)$ isomorphic to a fixed Lie group and $\gamma \in \G(k)$ is strongly regular, satisfies $m(\gamma) \le R$ and generates a compact subgroup at all non-archimedean places. By Lemma \ref{charpoly_count} it suffices to give a uniform estimate for $|\mathcal S_U(\gamma)|$. 


Define the following sets of finite places: 
\begin{align*}
  S_1 &= \{ v \in \val_f :\: |\Delta(\gamma)|_v \not= 1 \} ; \\
  S_2 &= \{ v \in \val_f :\: \G_\gamma \text{ is not split by an unramified } L_v/k_v \} 
\end{align*}
and let $L$ be the minimal Galois extension of $k$ splitting $\G_\gamma$. 
By Lemma \ref{l-GoodConjIneq1}, Lemma \ref{l-TateShafarevichBound}, \eqref{product_real_places}, Proposition \ref{p-GConjugacyLocalUnr} and Lemma \ref{c-ConjBoundNA}, we have that
\begin{equation}\label{all_lemmas}
  |\mathcal S_U(\gamma)| \ll C^{|S_1| + |S_2|} [L:k]^{\dim\G_\gamma[L:k]^2}
\end{equation}
for some $C$ depending only on $G$. Therefore, we must prove that $|S_1|, |S_2|, [L:k] = O_{R, G}(1)$.

For $S_1$ this is easy: we have $\Delta(\gamma) \in \mo_k$, so 
\[
2^{|S_1|} \le \prod_{v \in S_1} q_v \le N_{k/\QQ}\Delta(\gamma) \le e^R.
\]
In particular, $|S_1| \le R/\log(2)$.

On the other hand, by Proposition \ref{bound_disc_splitting_field}, we have that $N(\Delta_{L/k})$ is bounded by a constant depending only on $R, d$ and the absolute type of $\G$, so $[L:k]$ must be as well. Since $\prod_{v \in S_2} \mathfrak p_v$ divides $\Delta_{L/k}$, it also follows that $|S_2|$ is bounded. 

We have proven that all terms on the right of \eqref{all_lemmas} are bounded independently of $k, \G, \gamma$ as above, hence the theorem holds.


\section{Betti numbers} \label{s:quantbetti}

\subsection{Limit multiplicities and Betti numbers}

By Matsushima's formula, there exists a finite number of irreducible representations that control the size of the Betti numbers of $\Gamma \bs X$. More precisely, for every $i = 1, 2 , \dots, \dim{X}$, there exists a finite collection of unitary irreducible representations $\{\pi_j^{i}\}$ such that the $i$-th Betti number

$$b_i(\Gamma \bs X) = \dim(H^i(\Gamma \bs X)) = \sum_j  n(\pi_j^i, i) m ( \pi_j^{i}, \Gamma) $$
where $m ( \pi_j^{i}, \Gamma)$ is the multiplicity of $\pi_j^{i}$ in  $L^2(G/\Gamma)$. In addition, if $i \not= \tfrac 1 2 \dim(X)$ then none of the $\pi_j^i$ is a discrete series. Consequently, Theorem \ref{quant_betti_numbers} follows from the following limit multiplicities result. 

\begin{theorem}\label{estimatemultiplicity}
  Let $G$ be a semi-simple real Lie group and let $\pi$ be an irreducible unitary representation of $G$ which is not discrete series. Then, for any co-compact arithmetic lattice $\Gamma\subset G$ with the trace field $k$ we have 
  \[
  m(\pi,\Gamma) \ll_{\pi} \begin{cases} \vol(\Gamma\bs G)[k:\QQ]^{-1}& \text{ if }\pi\text{ is tempered,}\\
    \vol(\Gamma\bs G)e^{-c_{\pi}[k:\QQ]}&\text{ if }\pi\text{ in non-tempered,}\end{cases}
  \]
  where $c_{\pi}>0$ depends only on $\pi$.
\end{theorem}

In the rest of this section we prove this theorem. The proof is classical and follows the lines of \cite[10.2]{Fraczyk} and \cite[6.10-6.21]{7Sam}, using arguments originating with D.~Kazhdan and D.~L.~DeGeorge and N.~R.~Wallach \cite{DeWa78}. 



\subsection{Estimates for matrix coefficients} To apply the usual methods we need some well-known estimates for the growth of the $L^2$-norms of  matrix coefficients for non-discrete series representations restricted to a ball in $G$. We distinguish further between tempered and non-tempered representations; in the first case we will prove that.

\begin{proposition} \label{matrixcoeff_tempered}
  Let $\pi \in \widehat G$ be a tempered representation which is not a discrete series. Let $v \in \mathcal H_\pi$ be a $K$-finite vector, and $\phi(g) = \langle\pi(g)v, v\rangle$ the associated matrix coefficient. Then as $R \to +\infty$ we have:  
  \begin{equation}\label{estimate}
    \int_{B(R)} |\phi(g)|^2 \ dg  \gg_{\pi}  R^{2d + r} 
  \end{equation}
\end{proposition}

This inequality will follow from the asymptotic equivalent for matrix coefficients due to Harish-Chandra (see \ref{harish} below) and integration on $KAK$ decomposition of $G$. Using the same arguments we can also prove the following sharper result for non-tempered representations. 

\begin{proposition} \label{matrixcoeff_nontempered}
  Let $\pi \in \widehat G$ be a non-tempered representation . Let $v \in \mathcal H_\pi$ be a $K$-finite vector, and $\phi(g) = \langle\pi(g)v, v\rangle{\mathcal H_\pi}$ the associated matrix coefficient. Then there exits $a_\pi > 0$ such that as $R \to +\infty$ we have:  
  \begin{equation}\label{estimate_nontempered}
    \int_{B(R)} |\phi(g)|^2 \ dg  \gg_{\pi} e^{a_\pi R}.  
  \end{equation}
\end{proposition}


\subsubsection{Proof of Proposition \ref{matrixcoeff_tempered}}

We consider the Harish-Chandra expansion of matrix coefficients $\phi(g) = \langle \pi(g)v_0, v_1 \rangle$ as described in \cite[VIII.8]{Knapp2}. There exists a finite set of weights ${\lambda_1, \dots  \lambda_k} \subset \mathfrak{a}^{*}_\CC$ and non-zero functions $p_{i}: \liea \times H_0 \times H_0 \to \CC$, $\Lambda \in S $ such that for every $H \in \mathfrak{a}^{+}$:
\begin{equation}\label{harish} \lim_{t \to \infty} \frac{\sum_{i=1}^k e^{\lambda_i(tH)}P_i (tH, v_1, v_2) }{\langle\pi(\exp(tH))v_1, v_2 \rangle } = 1
\end{equation}
where $P_i (H, v_1, v_2)$ are polynomial in $H$, linear in $v_1$ and anti-linear in $v_2$. Moreover, the above limit converges uniformly if $H$ varies in a compact set in the interior of $\liea^{+}$. 

By \cite[Thm 8.53]{Knapp2} we have that $\text{Re}(\lambda_i) \leq  \rho$ for all $i$, where $\rho$ is half the sum of the simple roots of the lie algebra of $G$. Moreover as $\pi$ is tempered and not discrete series, none of its matrix coefficients can be square-integrable so equality must hold for some $i$ and for all nonzero $K$-finite vectors $v_1, v_2$. 


Let $p_1, p_2, \dots, p_m $ be the non-zero homogeneous polynomials in $\CC[\liea]$  of degree $d$ corresponding to the terms of degree $d$ in the polynomials $P_i(H, v_0, v_1)$ in \ref{harish}. Let $\alpha_1, \alpha_2, \dots, \alpha_m \in \liea^{*}_{\RR}$ be the imaginary parts of the corresponding roots $\lambda_1, \lambda_2, ..., \lambda_m$ as before and let $P(H) := \sum_{j=1}^m e^{i\langle H, \alpha_j \rangle} p_j(H)$. Observe that $e^{-\rho(H)}P(H)$ corresponds to the main term of the numerator in \ref{harish}.

Let $Q: \liea \to \RR^m$ the linear map defined for $H \in \liea$ by $$Q(H) = (\langle H, \alpha_1 \rangle, \dots \langle H, \alpha_m \rangle )$$ and let $q: \liea \to \mathbb{T}^m = \RR^{m}/ (2\pi)\ZZ^m$ be the induced map from $Q$. Let $T$ be the closure of $q(\liea)$ in $\mathbb{T}^k$, $T$ is a compact torus and $q(\liea)$ is dense in $T$. We will use the fact that $q(\liea)$ equidistributes in $T$, more precisely, we will make use of the following basic fact which can be easily proved by computing Fourier coefficients:

\begin{proposition}\label{equid}
  Let $f$ be a smooth $C^{\infty}$ function in $T$, then for any $H_0 \in \liea$, we have
  \[
  \lim_{R\to \infty} \frac{1}{\vol(B_{H_0}(R))}\int_{B_{H_0}(R)} f(q(H)) \ dH \to \int_{T} f
  \]
  Moreover the convergence is uniform over $H_0 \in \liea$.
\end{proposition}

Let $H_0$ be a nonzero vector in the interior of $\liea^{+}$ such that $P(H_0) \neq 0$. Let $\delta_1, \delta_2 > 0$ be small enough, so that the  closure of $B_{H_0}(\delta_1)$ is contained in the interior of $\liea^{+}$ and for all $H \in B_{H_0}(\delta_1)$ we have $|P(H)| \geq \delta_2$.

We have 
\begin{equation}\label{mainexp}P(tH)= t^d \sum_{j=1}^m e^{i\langle tH, \alpha_j \rangle} p_j(H). \end{equation} 
Let $x_0 = q(H_0) = (\langle H_0, \alpha_1 \rangle, \dots \langle H_0, \alpha_m \rangle )$  be the projection of $H_0$ in $T \subset \mathbb{T}^m$ and choose $B_{x_0}$ a ball around $x_0 \in T$ small enough so we can guarantee by \ref{mainexp} that for all $H \in B_{H_0}(\delta_1)$ and all $t >0$ such that $q(tH) \in B_{x_0} $ we have $|P(tH)| \geq \frac{1}{2}t^d \delta$.

Therefore for all $H \in B_{H_0}(\delta_1)$ and all $t >0$ 
\begin{equation}\label{wecare}
|\phi(k_1\exp(tH)k_2)| \gg  e^{-\rho(tH)}t^d 
\end{equation}
and computing integrals in $KAK$ decomposition (see  \cite[Prop. 5.28]{Knapp2}) we get :
\[
\int_{B(R)} |\phi(g)|^2 \ dg \gg \int_K\int_K \int_{ \liea^{+}}^{\rho(H) < R}  e^{2\rho(H)}|\phi(k_1\exp(H)k_2)|^2  \ dH d{k_1} d{k_2}.  
\]
We choose $\delta_2$ small enough such that 
\[
B_{R, H_0} := (\delta_2R)\B_{H_0}(\delta_2) \subseteq \{H \in \liea^{+} | \rho(H) < R \}
\]
and therefore 
\[
\int_{\liea^{+}}^{\rho(H) < R}  e^{\rho(H)}|\phi(k_1\exp(H)k_2)|^2  \ dH  \gg   \int_{B_{R, H_0}} e^{2\rho(H)}|\phi(k_1\exp(H)k_2)|^2  \ dH 
\]
Observe that $B_{R, H_0}$ is a ball of radius $O(R)$ which is contained in the cone $\{tH | t>0, H \in \B_{H_0}(\delta_2)\}$ and therefore the estimate \ref{wecare} together with proposition \ref{equid} (by taking $f$ to be a smooth approximation of the characteristic function of $B_{x_o}$) give us: 
\[
\frac{1}{\vol (B_{R, H_0})} \int_{B_{R, H_0} \cap q^{-1}(B_{x_0})} e^{2\rho(H)}|\phi(k_1\exp(H)k_2)|^2  \ dH \gg R^{2d} \vol(B_{x_0}) 
\]
which easily implies \eqref{estimate}.

\subsubsection{Proof of Proposition \ref{matrixcoeff_nontempered}} The argument is the same, except that as $\pi$ is not tempered one of the $\lambda_i$ appearing in \eqref{harish} must satisfy $\mathrm{Re}(\lambda_i) > \rho$, by \cite[Theorem 8.53]{Knapp2}. By using the same argument as for \eqref{wecare} we get that
\begin{equation}\label{wecare2}
|\phi(k_1\exp(tH)k_2)| \gg  e^{-(\rho+\eps)(tH)}t^d 
\end{equation}
where $\eps \in \mathfrak a_\RR^*$ is positive on $\mathfrak a^+$. Using the same integration scheme, we get that
\[
\frac{1}{\vol (B_{R, H_0})}
\int_{B_{R, H_0} \cap q^{-1}(B_{x_0})} e^{2\rho(H)}|\phi(k_1\exp(H)k_2)|^2  \ dH \gg e^{-aR} \vol(B_{x_0})
\]
for some $a > 0$ (depending on $\pi$ via $\eps$) so that \eqref{estimate_nontempered} follows.


\subsection{Proof of limit multiplicities}


The proof of Theorem \ref{estimatemultiplicity} follows well-known lines as well. By an argument of G.~Savin \cite{Savi89} (see also \cite[Lemma 6.15]{7Sam}), there exists a subspace $W \subset L^2(\Gamma \bs G)$ such that if
\[
\beta(x) = \sup_{f \in W,\, \|f\| = 1} |f(x)|^2
\]
then
\begin{equation} \label{kernmult}
  m(\pi, \Gamma) = \int_{\Gamma \bs G} \beta(x) dx.
\end{equation}
On the other hand there is the following estimate for $\beta(x)$ (see \cite[(6.21.1)]{7Sam}: there exists a vector $v \in \mathcal H_\pi$ (from which $W$ is defined) such that for any $r > 0$ we have for any $x \in \Gamma \bs G$ that 
\begin{equation} \label{estB}
  \beta(x) \le \frac 1{\|\phi_r\|^2} N_\Gamma(x, 2r)
\end{equation}
where
\[
\phi_r(g) = 1_{B(r)}(g)\langle\pi(g)v, v\rangle_{\mathcal H_\pi}
\]
and
\[
N_\Gamma(x, \ell) = |\{\gamma \in \Gamma : d(x, \gamma x) < \ell\}|.
\]
By Proposition 5.1 in \cite{FHR_complexity}, there exists $m \in \NN$ and $\eps > 0$ such that 
\[
N_\Gamma(x, 2r) \le [k:\QQ]^m
\]
for any $r \le \tfrac{\eps}2 [k : \QQ]$. By the argument in the proof of \cite[Lemma 6.20]{7Sam}, it follows that for such $r$ we have 
\[
\int_{\Gamma\bs G} N_\Gamma(x, 2r) dx \le [k:\QQ]^m\vol \left( (\Gamma \bs G)_{\le 2r} \right) + \vol(\Gamma \bs G)
\]
and by Theorem \ref{unbounded_degree} this is $(1+o(1))\vol(\Gamma \bs G)$ as $[k : \QQ] \to +\infty$. Using \eqref{kernmult} and \eqref{estB} with $r = \tfrac{\eps}2 [k : \QQ]$ we get
\[
m(\pi, \Gamma) \ll \frac 1{\|\phi_r\|^2} \vol(\Gamma \bs G)
\]
and now Theorem \ref{estimatemultiplicity} follows from Proposition \ref{matrixcoeff_tempered} (when $\pi$ is tempered) and Proposition \ref{matrixcoeff_nontempered} (when it is not).




\bibliographystyle{plain}
\bibliography{bib}

\end{document}